\documentclass[12pt]{amsart}
\usepackage{amssymb,amsmath,amsthm}
\usepackage{epsfig}

\newtheorem{theorem}{Theorem}[section]
\newtheorem{lemma}[theorem]{Lemma}
\newtheorem{remark}[theorem]{Remark}
\newtheorem{definition}[theorem]{Definition}

\newtheorem{proposition}[theorem]{Proposition}
\newtheorem{corollary}[theorem]{Corollary}

\newtheorem{observation}[theorem]{Observation}

\newcommand\Z{{\mathbb{Z}}}
\newcommand\R{{\mathbb{R}}}
\title{Cocycle invariants of  codimension 2-embeddings of manifolds }
\author{
J\'ozef H. Przytycki, Witold Rosicki}
\begin{document}
\maketitle
\centerline{December 27, 2006 -- October 11, 2013}
\begin{quotation}
ABSTRACT. \baselineskip=10pt
We consider the classical problem of a position of $n$-dimensional
manifold $M^{n}$ in  $\R^{n+2}$.\\
We show that we can define the fundamental $(n+1)$-cycle and the shadow 
fundamental $(n+2)$-cycle for a fundamental quandle of a knotting $M^n \to \R^{n+2}$.\\
In particular, we show that for any fixed quandle, quandle 
coloring, and shadow quandle coloring, of a diagram of 
$M^n$ embedded in  $\R^{n+2}$ we have $(n+1)$- and $(n+2)$-(co)cycle invariants   
(i.e. invariant under Roseman moves). 
\\
\end{quotation}
\ \\
\ \\

\tableofcontents
\section{Introduction}
\markboth{\hfil{\sc J\'ozef H. Przytycki and Witold Rosicki }\hfil}
{\hfil{\sc Cocycle invariants of codimension 2 embeddings}\hfil}

We consider the classical problem of a position of $n$-dimensional 
manifold $M^{n}$, in an $(n+2)$-manifold $W^{n+2}$. 
The classical case deals with $W^{n+2}=\R^{n+2}$, however our method is 
also well suited for a more general case of 
 $W^{n+2}$ being a product of 
an oriented $(n+1)$-manifold $F^{n+1}$
and the interval or the twisted interval bundle over an 
unorientable $(n+1)$-manifold $F^{n+1}$, as we have, in these cases,
 the natural projection of $W^{n+2}$ onto $F^{n+1}$ (see Section \ref{Section 7}). 

Historically, the main tool to study  an $n$-knotting, $f:M^{n} \to \R^{n+2}$,
 was the fundamental group of the knotting complement in $\R^{n+2}$. 
This was greatly extended by applying quandle colorings and (co)cycle 
invariants. We follow, to some extent, the exposition by Carter, Kamada, and Saito in \cite{CKS-3},
generalizing on the way the case of surfaces in $\R^4$ to general $n$-knottings.
The paper is organized as follows:
at the beginning of the first section we recall the definition of a rack and quandle and their (co)homology.
Then we give a short introduction to diagrams, $D_M$, of knottings, and rack and quandle colorings 
of $D_M$. Furthermore, we analyze shadow rack and quandle colorings. In the second part of the first 
section we define $(n+1)$- and $(n+2)$-chains associated to rack and quandle colorings.
In the second section we prove that our chains are, in fact, cycles.
In the third section we show that the set of colorings by a given quandle is a topological invariant.
The first step in this direction is given by comparing two definitions of the fundamental rack and 
quandle of a knotting (one from the diagram and one abstract).
In the fourth section we show that the homology classes represented by cycles constructed for diagrams of knottings
are topological invariants. Here we carefully consider Roseman's pass move (generalized third Reidemeister move).
We offer various versions of cycle invariants of knottings in particular taking into account the fact that 
a quandle acts on the space of quandle colorings. We complete Section \ref{Section 4} by expressing 
invariants in the language of cohomology (cocycle invariants). 
In the fifth section we generalize previous results to twisted homology and cohomology.
In the sixth section we discuss in detail general position projection of $n$-knotting and 
Roseman moves; this is a service section to the previous considerations. 

Finally, in Section \ref{Section 7} we give a short overview of a knotting in $F^{n+1}\bar\times [0,1]$, 
and in Section \ref{Section 8} we discuss possibility of working with a Yang-Baxter operator in place 
of a right self-distributive operation.

  \subsection{Quandles and quandle homology}
We give here a short historical introduction to distributive structures and to homology based 
on distributivity.  
The word {\it distributivity} was coined in 1814 by Francois Servois. 
C.S. Peirce in 1880  emphasized the importance of (right) self-distributivity in algebraic
structures \cite{Peir}. The first explicit example of a non-associative self-distributive system was given 
by Ernst Schr\"oder in 1887 \cite{Schr,Deh}.
The detailed study of distributive structures started with the 1929 paper by 
C.~Burstin and W.~Mayer\footnote{Walter Mayer is well known for Mayer-Vietoris sequence and for being 
assistant to Einstein at Institute for Advanced Study, Princeton.} \cite{B-M}. 
The first book partially devoted to distributivity is by Anton Sushkevich, 1937 \cite{Sus}. 

\begin{definition}\label{Definition 1.1} Let $(X;*)$ be a magma, that is a set with binary operation, then:
\begin{enumerate}
\item[(i)] If $*$ is right self-distributive, that is, $(a*b)*c=(a*c)*(b*c)$,
then $(X;*)$ is called a RDS or a shelf (the term coined by Alissa Crans  in her
PhD thesis \cite{Cra}).
\item[(ii)] If a shelf $(X;*)$ satisfies the idempotent\footnote{The term coined in 1870 by 
Benjamin Peirce \cite{Pei}, the father of Charles Sanders Peirce.}
 condition, $a*a=a$ for any $a\in X$, then it
is called a {\it right spindle}, or just a spindle (again the term coined by Crans).
\item[(iii)] If a shelf $(X;*)$ has $*$ invertible,\footnote{If $X$ is a set then the set $Bin(X)$ of all 
binary operations on $X$ forms a monoid with composition 
$*_1*_2$ given by $a*_1*_2b=(a*_1b)*_2b$ and the identity element 
$*_0$ given by $a*_0b=a$. Then the condition (ii) is equivalent to invertibility of $*$ in $Bin(X)$. 
If $*$ is invertible, we write $\bar *$ for $*^{-1}$.} 
that is the map $*_b: X \to X$ given by $*_b(a)=a*b$ is 
a bijection for any $b\in X$, then it is called a {\it rack}\footnote{The term wrack, like 
in ``wrack and ruin", of J.H.Conway from 1959, was modified to rack in \cite{F-R}). 
The main example considered in 1959 by Conway and Wraith was a group $G$ with a $*$ operation given
by conjugation, that is, $a*b=b^{-1}ab$ \cite{C-W}.}.
\item[(iv)] If a rack $(X;*)$ satisfies the idempotent condition, then it is called a {\it quandle} (the term
coined in Joyce's PhD thesis of 1979 \cite{Joy-1}). 
\item[(v)] If a quandle $(X;*)$ satisfies $(a*b)*b=a$ then it is called  {\it kei} or
an involutive quandle. The term kei (\psfig{figure=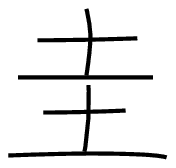,height=0.3cm})
was coined in a pioneering paper by M.Takasaki in 1942 \cite{Tak}\footnote{Mituhisa Takasaki worked at Harbin Technical 
University in 1940, 
likely as an assistant to K\^oshichi Toyoda. Both perished when Red army entered Harbin in 1945.
Takasaki was considering keis associated to abelian groups, that is
the Takasaki kei (or quandle) of an abelian group $H$, denoted by $T(H)$ satisfies $a*b= 2b-a$.}
\end{enumerate}
\end{definition}

Three axioms of a quandle arise (see \cite{Joy-2,Matv}) as an algebraic reflection of three Reidemeister moves 
on link diagrams. Idempotent condition corresponds to the first move, invertibility to the second, 
and right self-distributivity to the third move. In Figure 1.2 we illustrate how right self-distributivity 
is arising from the third Reidemeister move, $R_3$ when  we color (label) arcs of the diagram 
by elements of $X$ according to the following rule (Figure 1.1): \ \\ \ \\

 \centerline{\psfig{figure=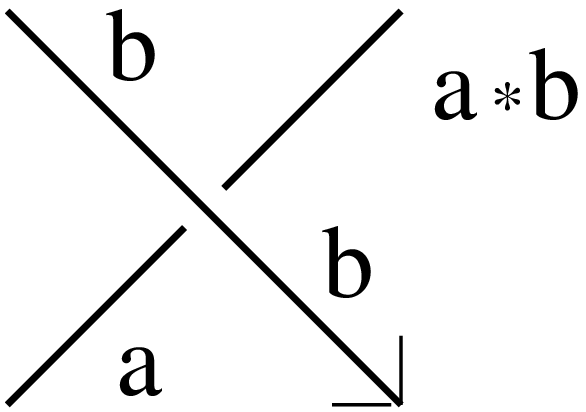,height=3.4cm}}.
\ \\ 
\centerline{Figure 1.1; magma coloring of a crossing}
\ \\ \ \\
\centerline{\psfig{figure=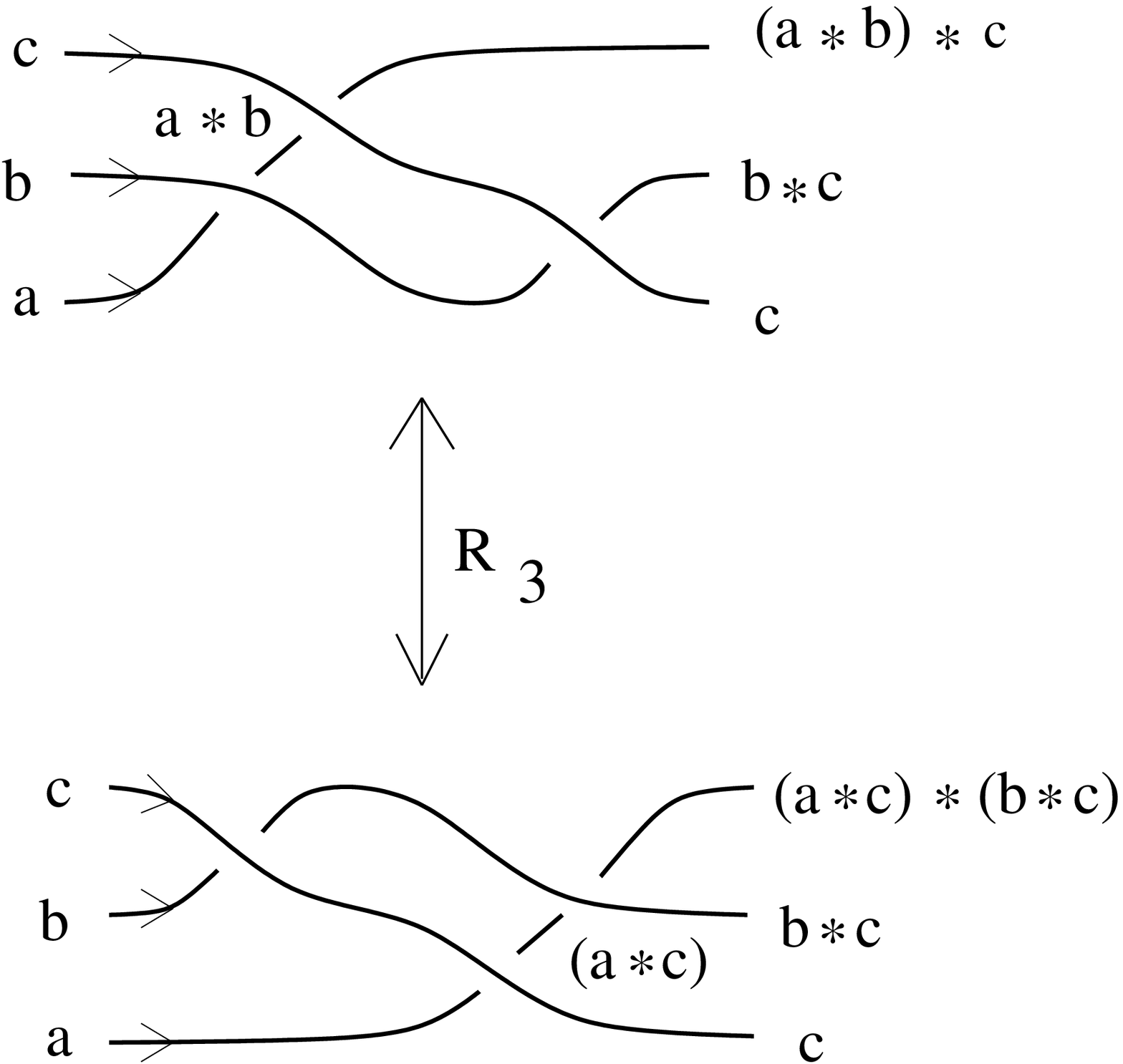,height=5.1cm}}
\\ \ \\
\centerline{Figure 1.2;  Distributivity from $R_3$}

(Co)homology of racks was introduced by Fenn, Rourke and Sanderson between  1990 and 1995, 
 \cite{FRS-1, Fenn}. Quandle (co)homology was constructed by Carter,
Kamada, and Saito (compare \cite{CKS-3}).
Their motivation was to associate to any link diagram an its quandle coloring elements (cocycles) of 
quandle cohomology. In \cite{CKS-3} it is done in details for knotting $f:M^n \to \R^{n+2}$ for $n=1$ or $2$, 
and we start our paper from the definition for any $n$ (essentially following \cite{CKS-3}).

We give here the definition of rack, degenerate and quandle (co)homology
after \cite{CKS-3}.
\begin{definition}\label{Definition 1.2}
\begin{enumerate}
\item[(i)]
For a given rack $X$ let $C^R_n(X)$ be the free abelian
group generated by $n$-tuples $(x_1,x_2,...,x_n)$ of elements of a rack $X$;
in other words $C^R_n(X) = {\Z}X^n = ({\Z}X)^{\otimes n}$. Define a boundary
homomorphism $\partial: C^R_n(X) \to C^R_{n-1}(X)$ by:
$$\partial(x_1,x_2,...,x_n) = $$
$$\sum_{i=2}^n (-1)^i((x_1,...,x_{i-1},x_{i+1},...,
x_n) - (x_1*x_i,x_2*x_i,...,x_{i-1}*x_i,x_{i+1},...,x_n)).$$
$(C^R_*(X),\partial)$ is called a rack chain complex of $X$.
\item[(ii)] Assume that $X$ is a quandle, then we have a subchain
complex $C^D_n(X) \subset C^R_n(X)$ generated by $n$-tuples $(x_1,...,x_n)$
with $x_{i+1}=x_i$ for some $i$. The subchain complex $(C^D_n(X),\partial)$
is called a degenerated chain complex of a quandle $X$.
\item[(iii)] The quotient chain complex $C^Q_n(X)=C^R_n(X)/C^D_n(X)$ is
called the quandle chain complex. We have the short exact sequence
of chain complexes:

$$ 0 \to C^D_n(X) \to C^R_n(X) \to C^Q_n(X)\to 0.$$
\item[(iv)] The homology of rack, degenerate and quandle chain complexes
are called rack, degenerate and quandle homology, respectively. We have
the long exact sequence of homology of quandles:
$$ ...\to H^D_n(X) \to H^R_n(X) \to H^Q_n(X)\to H^D_{n-1}(X) \to ...$$
\end{enumerate}
\end{definition}
R.~Litherland and S.~Nelson \cite{L-N}
proved that the short exact sequence from (iii) splits respecting
the chain maps. $\alpha: C^Q_n(X) \to C^R_n(X)$ is given, 
in the notation introduced in \cite{N-P}, by:
$$\alpha(x_1,x_2,x_3,...,x_n) = (x_1,x_2-x_1,x_3-x_2,...,x_n-x_{n-1}).$$
(recall that in our notation $(x_1,x_2-x_1,x_3-x_2,...,x_n - x_{n-1}) =
x_1\otimes (x_2-x_1)\otimes (x_3-x_2)\otimes \cdots \otimes
(x_n-x_{n-1}) \in C^R_n(X)$).
In particular, $\alpha$ is a chain complex monomorphism and
$H^R_n(X) = H^D_n(X) \oplus \alpha_*(H^Q_n(X))$.\\
In a recent paper \cite{P-P-2} it is demonstrated that degenerate homology of a quandle can 
be reconstructed from the quandle (normalized) homology of the quandle by a version of a K\"unneth formula.

We define cohomology in a standard way (we follow \cite{CKS-3}):
\begin{definition}\label{Definition 1.3}
For an abelian group $A$ define the cochain complexes $C^*_W(X,A)=Hom(C_*^W,A)$. Here, 
$W=D,R,Q$ so we describe all cases (degenerate, rack and quandle) . 
We define $\partial^n: C^n \to C^{n+1}$ in the usual way, that is for $c\in C^n_W(X;A)$ we have: 
$$\partial^n(c)((x_1,...,x_n,x_{n+1})= c(\partial_n((x_1,...,x_n,x_{n+1})).$$
Cohomology groups are defines as usual as $H^n_W(X,A)= ker \partial^n/im (\partial^{n-1})$.
\end{definition}

Another useful definition is that of right $X$-set, in particular the set of colorings of 
knotting diagram by a quandle $X$ will be a right $X$-quandle-set.
\begin{definition}\label{Definition 1.4} 
Let $E$ be a set, $(X;*)$ a magma and $*:E\times X \to E$ an action of $X$ on $E$
(we can use the same symbol $*$ for operation in $X$ and the action as it unlikely leads to confusion). Then
\begin{enumerate}
\item[(i)] If $(X;*)$ is a shelf and $(e*x_1)*x_2= (e*x_2)*(x_1*x_2)$ then $E$ is a right $X$-shelf-set.
\item[(ii)] If $(X;*)$ is a rack and $E$ a right $X$-shelf-set and additionally the map 
$*_b: E \to E$ given by $*_b(e)= e*b$ is invertible then we say that $E$ is a right $X$-rack-set.
In the case $X$ is a quandle we will say that $E$ is a right $X$-quandle-set.
\end{enumerate}
\end{definition}
The basic example of a quandle (resp. rack or shelf) right $X$-quandle- (resp. rack-, shelf-)-set 
is $E=X^n$ with $(x_1,...,x_n)*x= (x_1*x,...,x_n*x)$. One should also add that 
for a given $X$-rack-set $E$ one can define homology (as before) by assuming $C_n(X,E)= E\times X^n$, 
and $d_i^{(*)}(e,x_1,...,x_n)=(e*x_i,x_1*x_i,...,x_{i-1}*x_i,x_{i+1},...,x_n)$.
If $E$ has one point, so the action is trivial, we reach exactly homology of Definition \ref{Definition 1.2}(i).

\begin{observation}\label{Observation 1.5}
The chain groups $C_n(X)=ZX^n$ of a rack chain complex are $X$-rack-sets. Furthermore, the action 
$*_x: C_n(X) \to C_n(X)$, given by $*_x(x_1,...,x_n)= (x_1,...,x_n)*x=(x_1*x,...,x_n*x)$, 
is a chain map for any $x$, inducing the identity on homology. It is well know but important fact 
(see e.g. \cite{CJKS,N-P}) and we use it in Theorems \ref{Theorem 4.4}, \ref{Theorem 4.8}, and \ref{Theorem 5.5}. 
To prove this fact we use chain homotopy $(-1)^{n+1}h_x: C_n\to C_{n+1}$,
where $h_x(x_1,...,x_n)= (x_1,...,x_n,x)$. We check directly that 
$\partial_{n+1}(-1)^{n+1}h_x+ (-1)^{n}h_x\partial_{n}= Id - *_x$.
If we consider $\partial^T= t\partial^{(*_0)}-  \partial^{((*)}$, as is the case in twisted rack or quandle 
homology, we obtain that $*_x$ induces $t\cdot Id$ on homology. We use it in Definition \ref{Definition 5.3}.

\end{observation}

We need yet another observation that if $X$ is a quandle and $E$ is an $X$-quandle-set then $X\cup E$ has 
also a natural quandle structure.
\begin{observation}\label{Observation 1.6}
Let $(X;*)$ be a shelf and $E$ an  $X$-shelf-set (with a right action of $X$ on $E$ also denoted by $*$, then
$X\sqcup E$ is also a shelf with $*$ operation $a*y=a$ for any $a\in X\sqcup E$, and $y\in E$. 
Furthermore,
 if $(X;*)$ is a rack (resp. spindle, or quandle) then $(X\sqcup E;*)$ is also a rack (resp. spindle, or quandle).

Then we observe that a chain complex $C_n(X,E)$ is a subchain complex of $C_n(X\sqcup E)$.
\end{observation}
\subsection{Presimplicial module and a weak simplicial module}

We follow here \cite{Lod,Prz-1} and introduce here the notion of a presimplicial and weak simplicial module.
This will simplify our calculation and provide a language for visualization.

\begin{definition}\label{Definition 1.7}
A weak simplicial module $(M_n,d_i,s_i)$ is a collection of
$R$-modules $M_n$, $n\geq 0$, together with face maps, $d_i:M_n\to M_{n-1}$ and degenerate maps
$s_i: M_n\to M_{n+1}$, $0\leq i \leq n$, which satisfy the following properties:
$$ (1) \ \ \  d_id_j = d_{j-1}d_i\ for\ i<j. $$
$$(2)\ \ \ s_is_j=s_{j+1}s_i,\ \ 0\leq i \leq j \leq n, $$
$$ (3) \ \ \ d_is_j= \left\{ \begin{array}{rl}
 s_{j-1}d_i &\mbox{ if $i<j$} \\
s_{j}d_{i-1} &\mbox{ if $i>j+1$}
       \end{array} \right.
$$
$$ (4') \ \ \ d_is_i=d_{i+1}s_i. $$
$(M_n,d_i)$ satisfying (1) is called a presimplicial module and leads to the chain complex
$(M_n,\partial_n)$ with $\partial_n = \sum_{i=0}^n(-1)^id_i$.
\end{definition}
If (4') is replaced by a stronger condition\\
(4)\  $d_is_i=d_{i+1}s_i= Id_{M_n} $ then
$(M_n,d_i,s_i)$ is a (classical) simplicial module. 

The following basic lemma will be used later:
\begin{lemma}\label{Lemma 1.8}
Let $(M_n,d_i)$ be a  presimplicial module then the map $d_0d_0: C_n \to C_{n-2}$ is a chain map,
chain homotopic to zero. In particular, if $d_0d_0=0$ then $(-1)^nd_0$ is a chain map.
\end{lemma}
\begin{proof}
We think of $d_0: C_n \to C_{n-1}$ as a chain homotopy and we have:
$$d_0\partial_n + \partial_{n-1}d_0 = d_0\sum_{i=0}^n(-1)^id_i + \sum_{i=0}^{n-1}(-1)^id_id_0=
d_0d_0 + \sum_{i=1}^n(-1)^i(d_0d_i - d_{i-1}d_0)  $$
$$ \stackrel{(1)}{=} d_0d_0.$$
In particular, if $d_0d_0=0$ we have $(-1)^nd_0\partial_n = (-1)^{n-1}\partial_{n-1}d_0.$
\end{proof}

For us it is important that rack and quandle homology can be described in the language of 
weak simplicial modules:
\begin{proposition}(\cite{Prz-1})\label{Proposition 1.9}
\begin{enumerate}
\item[(i)]
Let $(X;*)$ be a rack, $C_n=ZX^n$,  $d^{(*_0)}_i: C_n \to C_{n-1}$ is given by 
$d^{(*_0)}_i(x_1,...,x_n)=(x_1,...,x_{i-1},x_{i+1},...,x_n)$,
 and $d^{(*)}_i: C_n \to C_{n-1}$ is given by $d^{(*)}_i(x_1,...,x_n)=(x_1*x_i,...,x_{i-1}*x_i,x_{i+1},...,x_n)$,
and furthermore $d_i= d^{(*_0)} - d^{(*)}_i$,
 then $(C_n(X),d^{(*_0)}_i)$, $(C_n(X),d^{(*)}_i)$, and $(C_n(X),d_i)$ are presimplicial modules.
(we have here shift by one comparing to Definition \ref{Definition 1.7}, that is we start from 1 not from 0,
 but it is not important in our considerations).
\item[(ii)] Assume now that $(X;*)$ is a quandle and degeneracy maps $s_i:C_n(X) \to C_{n+1}(X)$ is given, 
as before, by $s_i(x_1,...,x_n)=(x_1,...,x_{i-1},x_i,x_i,x_{i+1},...,x_n)$. Then
$(C_n(X),d^{(*_0)}_i,s_i)$, $(C_n(X),d^{(*)}_i,s_i)$, and $(C_n(X),d_i,s_i)$ are weak simplicial modules.
\end{enumerate}
\end{proposition}
\begin{remark}\label{Remark 1.10}
\begin{enumerate}
\item[(i)] The homology related to $(C_n(X),d^{(*)}_i)$ is called one term distributive homology and 
it is studied in \cite{Prz-1,P-S,P-P-1,P-P-2,CPP}.
\item[(ii)] We define the trivial quandle $(X;*_0)$ by $a*_0b=a$. Then indeed we have
$d^{(*_0)}_i((x_1,...,x_n)=(x_1*_0x_i,...,x_{i-1}*_0x_i,x_{i+1},...,x_n)=(x_1,...,x_{i-1},x_{i+1},...,x_n)$
\item[(iii)] Notice, that $d^{(*)}_1= d^{(*_0)}_1$ so $d_1=d^{(*_0)}_1-d^{(*)}_1=0$, and this is the reason 
why we could start summation in Definition \ref{Definition 1.2} from $i=2$ (but ideologically it 
may be better to start summation from $i=1$).
\item[(iv)] Let $\gamma_n = d_0^{(*)}: C_n \to C_{n-1}$, that is $\gamma(x_1,x_2,...,x_n)= (x_2,...,x_n)$,
then $(-1)^n\gamma_n$ is a chain map in $(C_n,\partial_n)$, by Lemma \ref{Lemma 1.8} (e.g. \cite{CJKS,N-P}.
\end{enumerate}
\end{remark}

\subsection{Codimension 2 embedding, lower decker set}
We introduce here, following \cite{CKS-3,Kam}, the language needed to define quandle colorings and (co)cycle 
invariants in codimension 2.\ 
Let $M=M^n$ be a closed smooth $n$-dimensional manifold and
 $f: M \to {\mathbb R}^{n+2}$ its smooth embedding which is called a smooth knotting (or just knotting). Define
$\pi: {\mathbb R}^{n+2}\to {\mathbb R}^{n+1}$ by $\pi(x_1,....,x_{n+1},x_{n+2})= (x_1,....,x_{n+1})$ to be a
projection on the first $n+1$ coordinates.
 The projection of the knotting is the set $M^*= \pi f (M)$.
 Crossing set (or singularity set) $D^*$ of the knotting, is the closure in $M^*$ of the set
of all points $x^*\in M^*$ such that $(\pi f )^{-1}(x^*)$ contains at least two points 
(that is $D^*= \mbox{closure}(\{y\in {\mathbb R}^{n+1}\ | \ | \ \pi^{-1}(y)\cap M| > 1\})$). \
We define the double point set $D= (\pi f )^{-1}(D^*)$ 
(or sometimes as $(\pi)^{-1}(D^*)$ if we need it to be a subspace of ${\mathbb R}^{n+2}$).
Let $f: M \to {\mathbb R}^{n+2}$ be a knotting which is in general position  with respect to the projection
$\pi: {\mathbb R}^{n+2} \to {\mathbb R}^{n+1}$. The precise definition is in Section \ref{Section 6} (Definition 6.1),
here we only use the basic notions:\\
 $D^*$ is $(n-1)$-dimensional stratified complex.
Its $(n-1)$-dimensional strata consists of transverse double points (double point set strata $\Delta^1$).
The crossing set $D^*$ divides $\pi f(M)$ into pieces. Each piece (connected component of
$\pi f(M)-D^*$) is an open $n$-manifold embedded in $\R^{n+1}$ consisting of regular points of $\pi f(M)$,
which is called open regular sheet. Regular sheets are 2-sided (even if we allow $M$ to be
nonorientable \cite{Kam}).

The lower decker set $D_-$ is the closure of the subset of pure double points which are
lower in the projection (that is with respect to the last coordinate of ${\mathbb R}^{n+2}$).
Similarly, the upper decker set $D_+$ is the closure of the subset of pure double points which are
higher in the projection. $M$ is cut by $D_-$ into the set of $n$-dimensional regions ($n$-regions)
 denoted by ${\mathcal R}$, that is
${\mathcal R}$ is the set of connected components of $M- D_-$ (notice that $\pi f$ restricted to $M- D_-$ 
is an embedding and that the image of an $n$-region can contain several open regular sheets).
 
The diagram $D_M$ of a knotting  with a general position projection is the knotting projection $M^*$
together with ``over under" information for the crossing set. 
In other words, it is $M^*$ with $D_+$ and $D_-$ given.

\subsection{Quandle colorings, and quandle shadow coloring}

We define here, after \cite{CKS-3}, the notion of quandle coloring and quandle shadow coloring 
of diagrams of knottings (\cite{CKS-3} gives only definition in dimension $n\leq 2$ but generalization 
is natural and it is implicit in the work of Fenn, Rourke and Sanderson \cite{F-R,FRS-2}).
In \cite{P-R} the core coloring was considered for any $n$.
 We assume in the paper (unless otherwise stated)
that the considered $n$-dimensional manifold $M$ is oriented thus the normal orientation (co-orientation) of every 
open sheet of $M^*$ is well defined\footnote{In the case of $M$ unorientable, we can work with involutive quandle (kei) 
$X$ and develop the theory of colorings and (co)-cycle invariants.}. 

\begin{definition}(Magma coloring)\label{Definition 1.11}
Fix a magma $(X;*)$.  
Let $f: M \to \R^{n+2}$ be an $n$-knotting, $\pi:\R^{n+2}\to \R^{n+1}$ a regular projection, and $D_M$ the
knotting diagram.
Let ${\mathcal R}$ be the set of $n$-regions of $M$ cut by lower decker set.
We define a  magma coloring of a diagram $D_M$ (or a pair $(M,\pi)$) as a function 
$\phi: {\mathcal R} \to X$ satisfying the following condition:
if $R_1$ and $R_2$ are two regions separated by
$n$-dimensional upper decker region $R_3$ and the orientation normal to $R_3$
 points from $R_1$ to $R_2$, then $\phi(R_1)* \phi(R_3)=  \phi(R_2)$; compare Figure 1.3.
Coloring of $n$-regions leads also to coloring of open sheets of the diagram $D_M$ of the knotting.
Through the paper we  often  refer to this as coloring of a knotting diagram.
We denote by $Col_X(D_M)$ the set of colorings of $D_M$ by $X$, and by $col_X(D_M)$ its 
cardinality. 
Note that the definition is not using any properties of $*$; only when we will demand 
invariance of $col_X(D_M)$ under various moves on $D_M$, we will need some specific properties of $*$. 
\end{definition}
\ \\
\centerline{\psfig{figure=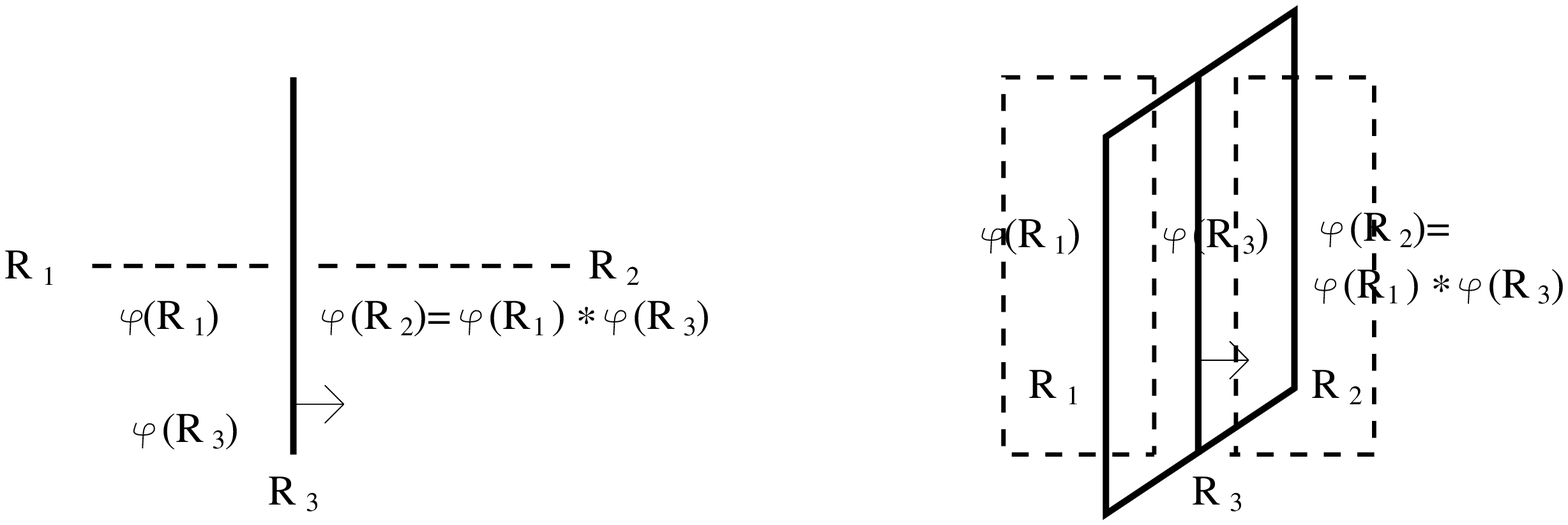,height=4.5cm}} \ \\
\centerline{Figure 1.3; rules for magma (e.g. quandle) coloring for $n=1,2$}
\ \\

If we assume that $(X;*)$ is a shelf then the set $Col_X(D_M)$ is a right $X$-shelf-space with 
an action of $X$ on $Col_X(D_M)$ given by 
$(\phi * x)(R)= \phi(R)*x$ for any region $R\in {\mathcal R}$. By the right self-distributivity of $*$ we have 
$(\phi *x)(R_1)* (\phi *x)(R_3)= (\phi (R_1)*x)*(\phi (R_3)*x) 
\stackrel{distr}{=}
 (\phi (R_1)*\phi (R_3))*x=
\phi(R_2)*x = (\phi*x)(R_2)$.

Before we define shadow coloring it is useful to notice that rack coloring of sheets of $D_M$ allows 
unique coloring of any closed path in $\R^{n+1}$ in a general position to $D_M$, as long as a base point 
is colored (Lemma \ref{Lemma 1.12}). In a preparation for the lemma we need the following:\\
Fix a rack $(X,*)$ and an element $q_0\in X$. Let $t_0 < t_1< ...<t_k <t_{k+1}$ be points on the line $\R$.
Each point $t_i$ ($1\leq i \leq k$) is equipped with a $\pm 1$ framing, according to the convention: 
\parbox{1.7cm}{\psfig{figure=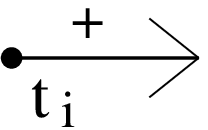,height=0.9cm}}, 
\parbox{1.7cm}{\psfig{figure=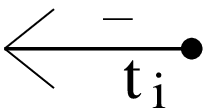,height=0.9cm}}.
Then any function $\phi: \{t_1,...,t_k \} \to X$ extends uniquely to the function $\tilde\phi:[t_0,t_k]\to X$ 
with $\tilde\phi(t_0)=q_0$ by the following rule. If $a\in [t_{i-1},t_i]$ and $b\in [t_i,t_{i+1}]$ then
$$ \tilde\phi(b)= \left \{ \begin{array}{ll}
                    \tilde\phi(a)*\phi(t_i) & \mbox{if the framing at $t_i$ is positive}\\
                    \tilde\phi(a)\bar *\phi(t_i) & \mbox{if the framing at $t_i$ is negative}
\end{array}
              \right.  $$ 
In particular, $\tilde\phi(t_{k+1})= (q_0*_1\phi(t_1))*_2...*_{k}\phi(t_k)$, where 
$*_i=*$ if the framing of $t_i$ is positive and $*_i=\bar *$ if the framing of $t_i$ is negative.
Finally we can apply the above to an arc $\alpha: [t_0,t_{k+1}] \to \R^{n+1}$ in a general position with 
respect to $D_M$ with some $X$-coloring $\phi$, and which cuts $D_M$ at $k$ points $\alpha(t_1),...,\alpha(t_k)$.
The framing of points $t_i$ ($1\leq i \leq k$) is yielded by co-orientation of $D_M$ and points 
$\alpha(t_i)$. We can identify $\phi(t)$ with $\phi(\alpha(t_i))$, thus by above $\phi$ can be extended 
to the function $\tilde\phi: [t_0,t_{k+1}] \to X$. Now we are ready to prove that:

\begin{lemma}\label{Lemma 1.12}
If $\alpha: [t_0,t_{k+1}]\to \R^{n+1}$ is a closed path, that is $\alpha(t_0)=\alpha(t_{k+1})$, then 
$\tilde\phi(\alpha(t_0))= \tilde\phi(\alpha(t_{k+1}))$. 

\end{lemma} 
\begin{proof} 
Using the fact that $\R^{n+1}$ is simple connected, we can contract $\alpha$ to a base point $\alpha(t_0)$,
and we can put contracting homotopy in a general position with respect to $D_M$. The proof is by induction on 
the number of critical points of  contracting homotopy. 
The critical points  are either cancelling 
a piece of the path going for and back, or crossing a double point strata. 
In the first case if we start from the color $a$ 
and cross color $b$, forth and back, thus we get a color $(a*b)\bar *b$ or $(a\bar *b)*b$ which is $a$ by invertibility 
of $*$. In the case when isotopy is crossing a double point set, we use the fact that double point crossing 
looks like classical crossing multiplied by $\R^{n-1}$, and the interesting case is when the closed path is 
below sheets it crosses. The situation can be illustrated by using classical crossing and 
coherence of coloring follows from right distributivity and invertibility of $*$, see Figures 1.4 and 1.5 .
\end{proof}
\ \\
\centerline{\psfig{figure=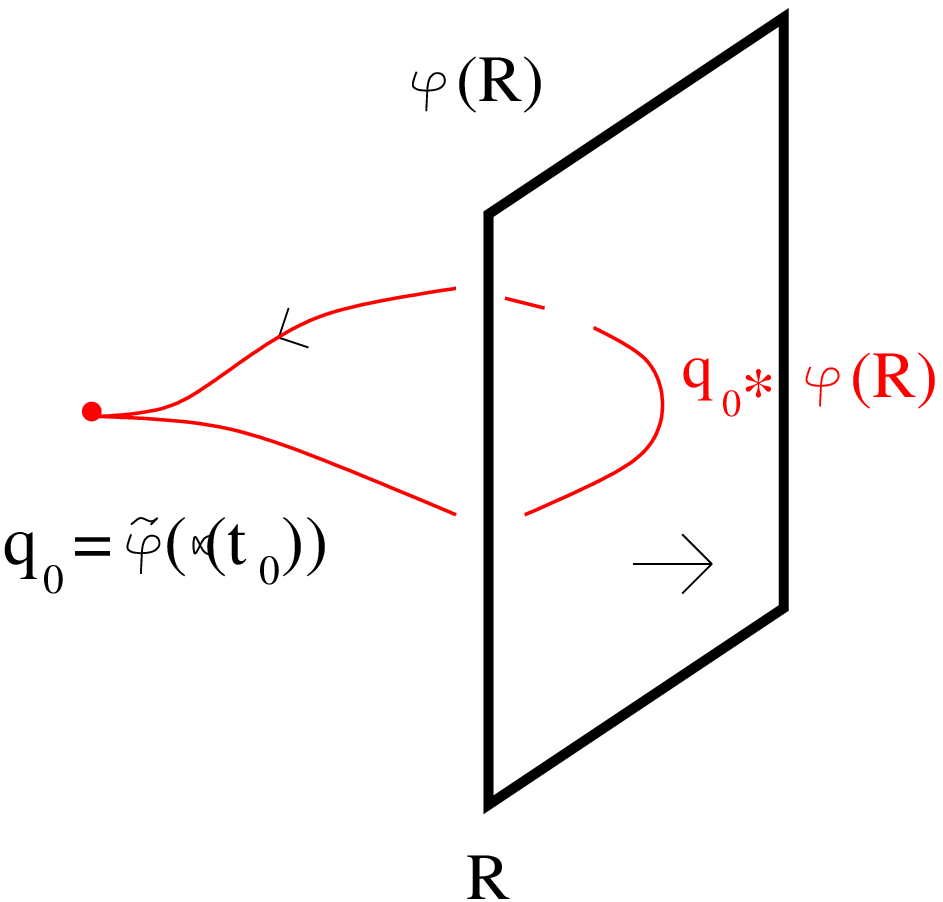,height=4.5cm}} \ \\
\centerline{Figure 1.4; a path moving for and back through the $n$-sheet }
\centerline{$(q_0*\phi(R))\bar *\phi(R)=q_0$}

\ \\
\centerline{\psfig{figure=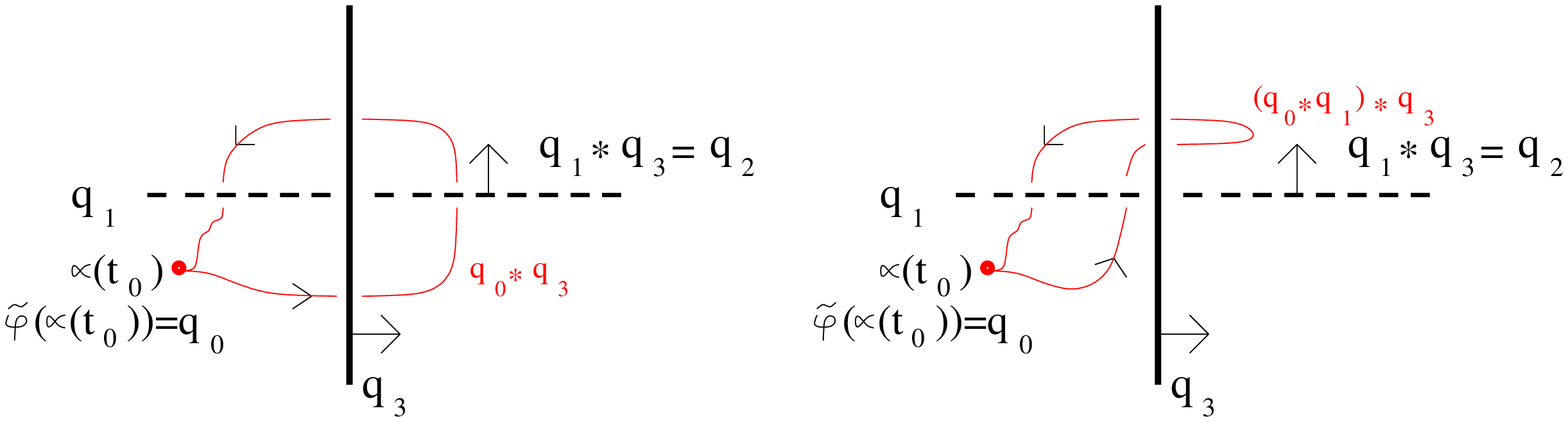,height=4.5cm}} \ \\
\centerline{Figure 1.5; a path is crossing double point stratum; $((q_0*q_3)*q_2)\bar *q_3)*\bar q_1=$ }
\centerline{$ ((q_0*q_3)*(q_1*q_3))\bar *q_3)*\bar q_1=
((q_0*q_1)*q_3)\bar *q_3)*\bar q_1= (q_0*q_1)*\bar q_1=q_0$}

\begin{definition}(Magma shadow coloring)\label{Definition 1.13}
A shadow coloring of a knotting diagram (extending the given coloring $\phi$) is a function
$\tilde \phi: \tilde {\mathcal R}\cup {\mathcal R} \to X$,
where $\tilde{\mathcal R}$ is the set of $(n+1)$-dimensional regions(chambers\footnote{We should
appreciate here, not that accidental, analogy to Weyl chambers in representation theory of 
Lie algebras.  Thus 
we use the term {\it chamber} throughout the paper.}) of
 ${\mathbb R}^{n+1}- \pi f(M)$
satisfying the following condition. If
$R_1$ and $R_2$ are $n+1$ regions (chambers) separated by $n$ dimensional region (regular sheet)
 $\alpha$ where the orientation normal
of $\alpha$  points from $R_1$ to $R_2$, then $\tilde \phi(R_1)* \tilde \phi(\alpha)= \tilde \phi(R_2)$
and $\tilde \phi$ restricted to the set of $n$-dimensional regions is a given coloring $\phi$ 
(compare \cite{CKS-3} and Figure 1.6 for $n=1 \mbox{ or } 2$). 
Again the definition works for any binary operation but if $(X;*)$ is a rack, 
then any coloring $\phi$ and  a constant $q_0$ chosen for a fixed $(n+1)$-chamber, $R_0$, 
yield the unique extension to shadow coloring $\tilde \phi$ 
so that  $\tilde \phi(R_0)=q_0$; this follows from Lemma \ref{Lemma 1.12}.\footnote{We can also  see 
this important property, as follows:
consider a small trivial circle $T_c$  in a chosen chamber $R_b$ of
 $\R^{n+1}- \pi f(M)$ let $T$ be the boundary of  a regular neighborhood of $T_c$,
thus $T= S^1 \times S^{n-1}$. We can extend coloring $\phi$ to $M\cup T$ by coloring $T$ 
by a fixed color $q_0$. Then we
isotope the circle $T_c$ (and $T$) so that it is always below $f(M)$ part but it touches
every chamber of $\R^{n+1}- \pi f(M)$.
Now having initial chamber colored by $q_0$ any other chamber is now colored by an appropriate color of
part of $T$ in the chamber. Unlike in approach using Lemma \ref{Lemma 1.12}, 
we use here the fact that Roseman moves on a  diagram 
are preserving colorings by a rack (see Theorem \ref{Theorem 3.4}).}.
We denote by $Col_{sh,X}(D_M)$ the set of all shadow colorings of $(\R^{n+1},D_M)$ by $X$ and by 
$col_{sh,X}(D_M)$ its cardinality.
\end{definition}

\ \\
\centerline{\psfig{figure=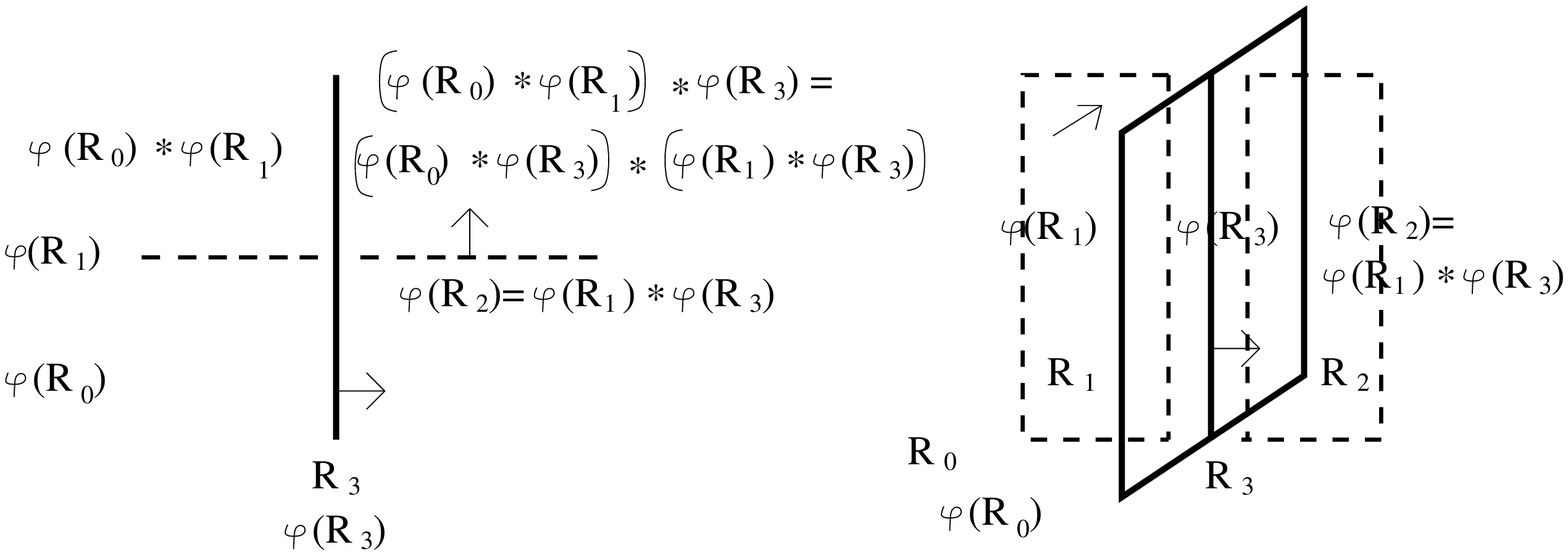,height=4.5cm}} \ \\
\centerline{Figure 1.6; Quandle shadow coloring for $n=1,2$}

Similarly as in non-shadow case if we assume that $(X;*)$ is a shelf then the set $Col_{sh,X}(D_M)$ 
is a right $X$-shelf-space with an action of $X$ on $Col_{sh,X}(D_M)$ given by
$(\tilde\phi * x)(R)= \tilde\phi(R)*x$ for any region or chamber $R$.

\begin{remark}\label{Remark 1.14}
A shadow coloring of $n$-knotting diagram can be interpreted as a special case of coloring in
dimension $n+1$. To this aim we consider the $(n+1)$ dimensional manifold $\tilde M= (M^n\times \R) \sqcup \R^{n+1}$ 
embedded in $\R^{n+3}$ as follows: Let $\tilde f: \tilde M \to \R^{n+3}$ with $\tilde f(m,x)=(f(m),x)$ and 
$\tilde f(x_1,...,x_{n+1})= (x_1,...,x_{n+1},h,0)$, where for $f(m)=(f_1(m),...,f_{n+2}(m))$ we assume 
$h \leq f_{n+2}(m)$ for any $m\in M^n$ (in other words, $\R^{n+1}$ is embedded below $M^n$).
The projection $\tilde\pi: \R^{n+3} \to \R^{n+2}$ is defined as $\pi \times Id$. We get the diagram 
$D_{\tilde M}= \tilde\pi \tilde f(\tilde M ) = (D_M \times \R) \cup (\R^{n+1}\times \{0\})$ in $\R^{n+2}$.
The points of multiplicity $n+1$ in $D_M$ gives rise to points of multiplicity $n+2$ in $D_{\tilde M}$, 
and each shadow coloring of $\R^{n+1},D_M)$ gives rise to a coloring of $D_{\tilde M}$
(also the (shadow) chain $c_{n+2}(D_M)$ gives rise to the chain $c_{n+2}(D_{\tilde M}$ as 
will be clear in the next subsection). 
 As we do not use Remark \ref{Remark 1.14} later, we should not worry that the resulting manifold $\tilde M$ 
is not compact (otherwise we need to consider knotting up to isotopy with compact support).
\end{remark}

\begin{observation}\label{Observation 1.15}
As noted by S.Kamada, one can consider shadow coloring $\tilde\phi$ also in the case of $(X;*)$ 
a rack and $E$ an $X$-rack-set.
In this case we color the chambers of $\R^{n+1} - D_M$ by elements of $E$ with a natural convention that 
if $R_1$ and $R_2$ are chambers separated by $n$ dimensional region $\alpha$ where the orientation normal
of $\alpha$  points from $R_1$ to $R_2$, then $\tilde \phi(R_1)* \tilde \phi(\alpha)= \tilde \phi(R_2)$.
One can also show, using Lemma \ref{Lemma 1.12} or Footnote 8, that
 if we take $\phi'(R_0)=q_0$ for some $q_0\in E$ then $\tilde\phi$ is uniquely 
extendable from coloring of $D_M$. The reason is that in place of $X$ and $E$ we can consider a rack $X\cup E$ 
as in Observation \ref{Observation 1.6}.
\end{observation}

\subsection{$(n+1)$ and $(n+2)$-chains of diagrams of knotting}


We show in this part how to any knotting diagram $D_M$ given by 
$\pi f: M^n \stackrel{f}{\rightarrow} \R^{n+2} \stackrel{\pi}{\rightarrow}
 \R^{n+1}$ and a shelf $(X;*)$ associate 
two chains of dimension $(n+1)$, and $(n+2)$ in rack (and quandle) chain groups
 $C^W_{n+1}(X)$ and $C^W_{n+2}(X)$ respectively ($W=R \mbox{ or } D \mbox{ or } Q$, that is 
Rack, Degenerate or Quandle).

We start from the classical theory in dimensions $n=1$.

Carter-Kamada-Saito noticed in 1998 that if we color a classical oriented link diagram, $D$, by elements of a given
quandle $X$ and consider a sum over  all crossings of $D$ of pairs in $X^2$, $\pm (q_1,q_2)$ according
to the  convention of Figure 1.7 then the sum has an interesting behavior under Reidemeister moves.
\ \\

\centerline{\psfig{figure=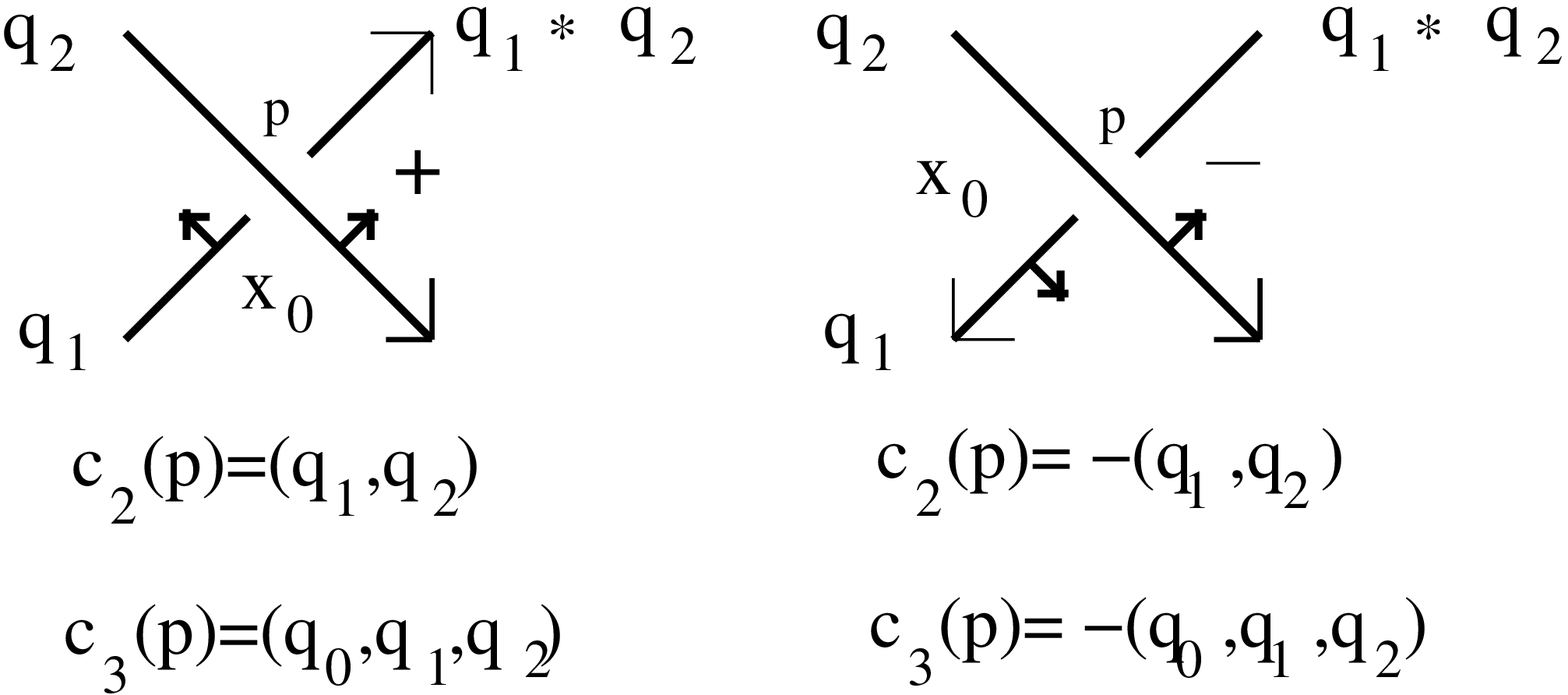,height=4.7cm}}
\centerline{Figure 1.7; The contribution to the 2-chain is $(q_1,q_2)$ for a positive crossing and}
 \centerline{ $-(q_1,q_2)$ for a negative crossing; in the case of the 3-chain for a shadow coloring, } 
\centerline{ we have $(q_0,q_1,q_2)$ and $-(q_0,q_1,q_2)$, respectively}
\ \\ \ \\
This lead them to define in 1998 a
2-cocycle invariant and relate it to rack homology defined between 1990 and 1995 by Fenn, Rourke, and Sanderson.
Because the first Reidemeister move is changing the sum by $\pm (x,x)$ they were assuming that
$(x,x)$ should be equivalent to zero (so $(x,x)$ should represent degenerate element). 
The 3-cocycle of the shadow $X$-colorings was motivated by \cite{R-S} and developed in \cite{CKS-1} 
(compare \cite{CKS-3}, page 154).
In particular, they noted that the 3-chain constructed with convention of Figure 1.7, that is $c_3(D)=
\sum_{p\in \mbox{crossings}} sgn(p)c_3(p)$, is a 3-cycle and Reidemeister moves preserve homology class 
of $c_3(D)$.

We define, in this subsection, the chains $c_{n+1}(D_M,\phi)$ and $c_{n+2}(D_M,\tilde\phi)$ 
for any diagram $D_M$ of a knotting $M$ and chosen colorings $\phi$ and $\tilde\phi$
 (the fact that they are cycles is proven in Section \ref{Section 2}, Theorem \ref{Theorem 2.1}, and topological
 invariance of their homology via Roseman moves is proven in Section 4, Theorem \ref{Theorem 4.1}). 
To make the general definition we need some conventions and notation concerning 
a crossing of multiplicity $(n+1)$ in a diagram of an $n$-knotting.\\
The sign of a crossing of multiplicity $(n+1)$  is chosen  so that it agrees with the definition 
of the sign of a crossing in a classical knot theory.
We say that the sign of $p$ is positive if $n+1$ normal vectors to $n+1$ hyperplanes 
intersecting at $p$ listed starting from the top (that is the normal vector to the highest hyperplane is first) 
form a positive orientation of ${\mathbb R}^{n+1}$; otherwise the sign of $p$ is equal to $-1$.
The source chamber $R_0$ of ${\mathbb R}^{n+1}-\pi f(M)$ adjacent to $p$ is the region 
from which normals of hyperplanes points (compare page 151 of \cite{CKS-3}).

\begin{definition}($(n+1$)-chain)\label{Definition 1.16}\
Let $f: M \to \R^{n+2}$ be an $n$-knotting, $\pi:\R^{n+2}\to \R^{n+1}$ a regular projection, and $D_M$ the 
knotting diagram. 
\begin{enumerate}
\item[(i)]
Fix a quandle $X$ and a coloring $\phi: {\mathcal R}\to X$,
Let $p$ be a crossing of multiplicity $n+1$ of $D_M$. We define a chain $c_{n+1}(p,\phi)\in C_{n+1}(X)$ by
$$ c_{n+1}(p) = sgn(p)(q_1,...,q_{n+1}) \mbox{ where } (q_1,...,q_{n+1}) \mbox{ is obtained as follows:} $$
We consider the source region, say $R_0$, around $p$ and $(q_1,...,q_{n+1})$ are colors of hyperplanes 
intersecting at $p$ around $R_0$ listed in the order of hyperplanes from the lowest to the highest 
(see Figure 1.8 for the case of $n=2$). We usually write $c_{n+1}(D_M)$ for $c_{n+1}(p,\phi)$ if $\phi$ 
is fixed.
\item[(ii)] The chain associated to the diagram and fixed coloring is the sum of above chains taken over 
all crossings of multiplicity $n+1$ of $D_M$:
$$c_{n+1}(D_M,\phi)=\sum_{p\in Crossings}c_{n+1}(p).$$
\item[(iii)] Finally, if $X$ is finite, we sum over all $X$ colorings of $D_M$ so the result is in the 
group ring over $C_{n+1}(X)$ (in fact, it is in the group ring of $H_{n+1}(X)$ but this will be 
proven later). It is convenient here to use multiplicative notation for chains so that 
$c_{n+1}(D_M,\phi)=\Pi_p(q_1,...,q_{n+1})^{sgn p}$ and then 
$$c_{n+1}(D_M)= \sum_{\phi}c_{n+1}(D_M,\phi)=
\sum_{\phi}\prod_p(q_1,...,q_{n+1})^{sgn (p)}.$$
\item[(iv)] If $X$ is possibly infinite, in place of a sum we consider the set with multiplicity
$$c^{set}_{n+1}(D_M)=\{c_{n+1}(D_M,\phi) \}_{\phi\in Col_X(D_M)}.$$
\end{enumerate}
\end{definition}
\begin{definition}($n+2$ (shadow) chain for $D_M$)\label{Definition 1.17}
Here we generalize the previous definition to construct $(n+2)$-chains from shadow colorings related to link diagram.
We color not only regions of $M$ (as in  Definitions \ref{Definition 1.11}, \ref{Definition 1.16})
but also $(n+1)$-chambers of ${\mathbb R}^{n+1}$ cut by $\pi f(M)$.
As before we take the product of signed chains associated to every multiplicity $(n+1)$-crossing point, 
$p$ and sum these products over all shadow colorings. Thus we start from $ c_{n+2}(p) = sgn(p)(q_0,q_1,...,q_{n+1})$,
where $q_0$ is the color of the source chamber $R_0$. In effect, if $X$ is finite then:
$$c_{n+2}(D_M)= \sum_{\tilde\phi}c_{n+2}(D_M,\tilde\phi)= \sum_{\tilde\phi}\prod_p(q_0,q_1,...,q_{n+1})^{sgn (p)}.$$
If $X$ is possibly infinite, in place of a sum we consider the set with multiplicity
$$c^{set}_{n+2}(D_M)=\{c_{n+2}(D_M,\tilde\phi) \}_{\tilde\phi\in Col_{sh,X}(D_M)}.$$

 We illustrate the case of $n=1$ in Figure 1.7, and the case of $n=2$ in Figure 1.8.
\end{definition}

\ \\
\centerline{\psfig{figure=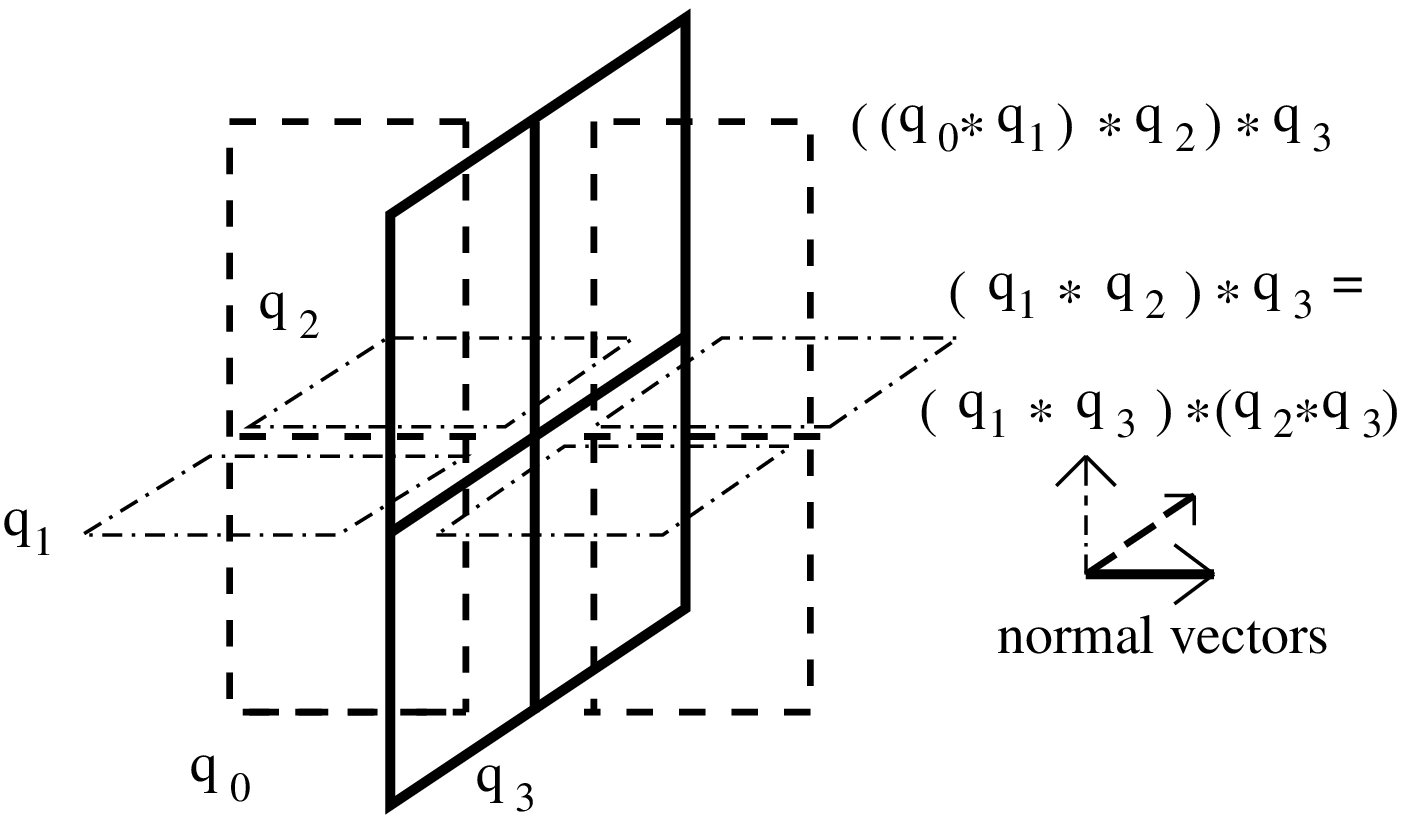,height=6.7cm}}
\centerline{Figure 1.8; multiplicity three point in $M^3$ knotting and quandle coloring;}
\centerline{the point yields the 3-chain $(q_1,q_2,q_3)$ and the 4-chain $(q_0,q_1,q_2,q_3)$;}
\centerline{ normal vectors yield here positive orientation}
\ \\

For $M$ which is not connected we can use a trick of \cite{CENS} to have more delicate chains.
They behave nicely under Roseman moves but they are not cycles. Thus we can use them to produce cocycle 
invariants but not cycle invariants of knottings (Theorem \ref{Theorem 4.8}(iii)). 

\begin{definition}\label{Definition 1.18}
Let $M= M_1 \cup M_2 \cup... \cup M_k$. We define an $(n+1)$-chain  $c_{n+1}(D_M,\phi,i)$  by considering
only those crossings of multiplicity $(n+1)$ whose bottom sheet belongs to $M_i$. We denote
the set of such $(n+1)$-crossings by ${\mathcal T}_i$.  Then we define
$$c_{n+1}(D_M,\phi,i) = \sum_{p\in {\mathcal T}_i}c_{n+1}(p,\phi).$$
\end{definition}


\section{$(n+1)$ and $(n+2)$ cycles for $n$-knotting diagrams}\label{Section 2}

We show in this section the two chains $c_{n+1}(D_M,\phi)$ and $c_{n+2}(D_M,\tilde\phi)$
 constructed in Subsection 1.5 are, in fact, cycles.
\begin{theorem}\label{Theorem 2.1}
The chains $c_{n+1}(D_M,\phi)$ and $c_{n+2}(D_M,\tilde\phi)$ are cycles in $C^Q_{n+1}(X)$ and $C^Q_{n+2}(X)$ 
respectively.
\end{theorem}
The main idea of the proof is to analyze points of multiplicity $n+1$ and $n+2$ in $M^*\in \R^{n+1}$, and 
associated $(n+1)$- and $(n+2)$-chains in $\Z X^{n+1}$ and $\Z X^{n+2}$, and to identify face maps of the chains,
$d_i^{(*_0)}$ and $d_i^{(*)}$, as associated to arcs of multiplicity $n$. Then by Roseman theory (see Section 
\ref{Section 6}), such an arc starts at a point of multiplicity $n+1$ (say $p_1$) and ends either 
at another point of multiplicity $n+1$ (say $p_2$) or at a singular point (of multiplicity less than $n$).
In the first case, we prove that there are proper cancellations of face maps associated to the arc. In the 
second case (which may happen for $n>1$), the colorings of the arc give degenerate chains 
which can be ignored in $C^Q_{n+1}(X)$ and $C^Q_{n+2}(X)$.
Details are given in the following subsections, starting from the classical case of $n=1$.

\subsection{Shadow 3-cycle for $n=1$}\label{Subsection 2.1}

We start from the known case of $n=1$ but present our proof in a way which will allow natural 
generalization for any $n$. 
We show in detail in this subsection that the 3-chain constructed with convention of Figure 1.7, that is $c_3(D)= 
\sum_{p\in \mbox{crossings}} sgn(p)c_3(p)$, is a 3-cycle. 

We start from 2 crossings $p_1$ and $p_2$ connected by an 
arc colored by the pair $q_{sh}(arc)=(q_0,b)$ in our convention (that is the color of the arc is $b$ and 
the shadow color of the source region (chamber) close to the arc is $q_0$; see Definition \ref{Definition 2.2}). 
In our examples the horizontal line 
is first above and then below the other arcs, and it will be denoted as pair of type $(2,1)$ later 
in generalization. For a reader who would like to visualize here the general case, we stress that 
the arc connecting crossings $p_1$ and $p_2$ will be an arc of points of multiplicity $n$ (intersection 
of $n$ hyperplanes in $\R^{n+1}$) connecting points of multiplicity $n+1$.

Now consider our four cases.
\ \\ \ \\
\centerline{\psfig{figure=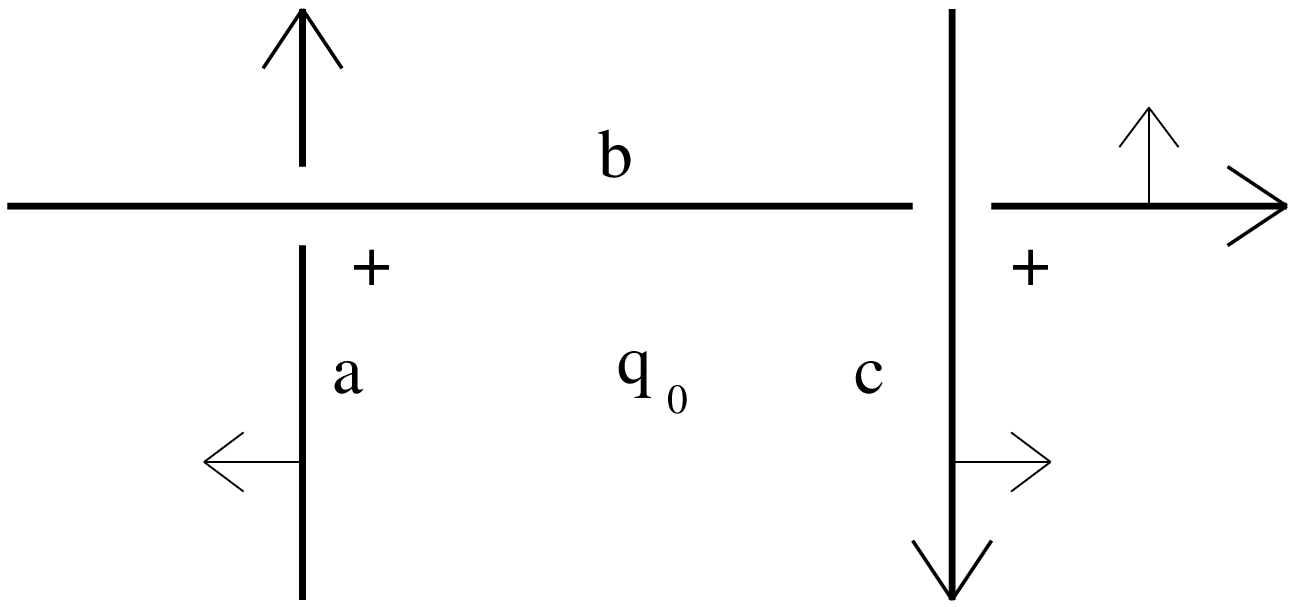,height=2.7cm}\ \ \ \ \ \ \
\psfig{figure=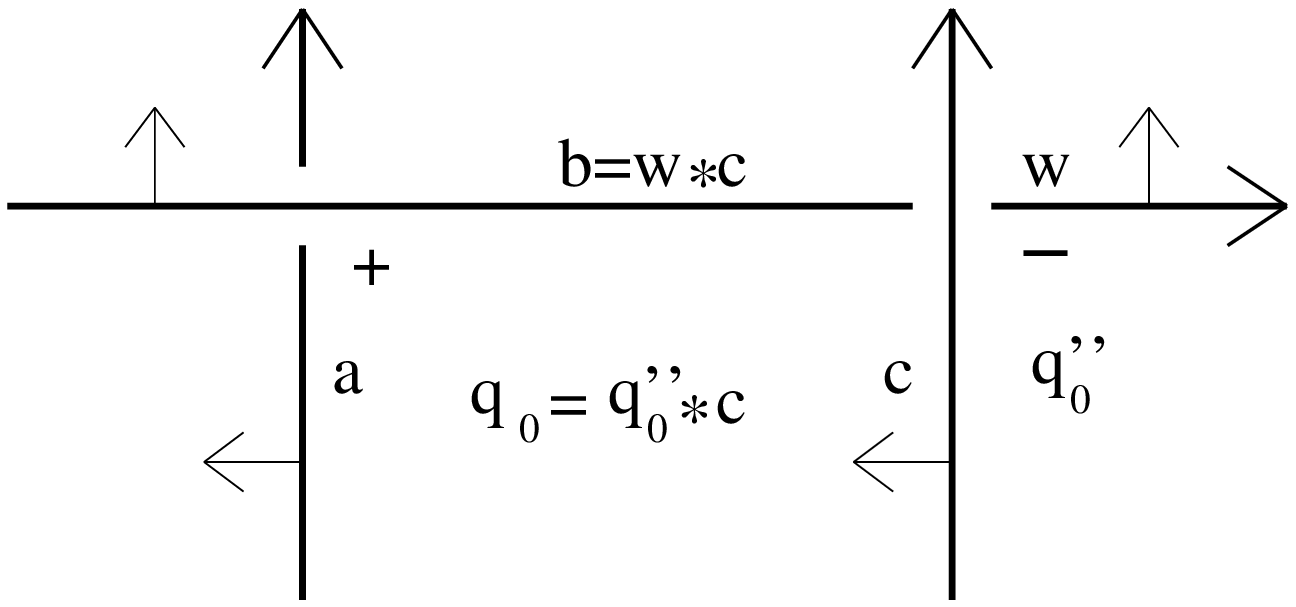,height=2.7cm}}
\ \\
\centerline{\psfig{figure=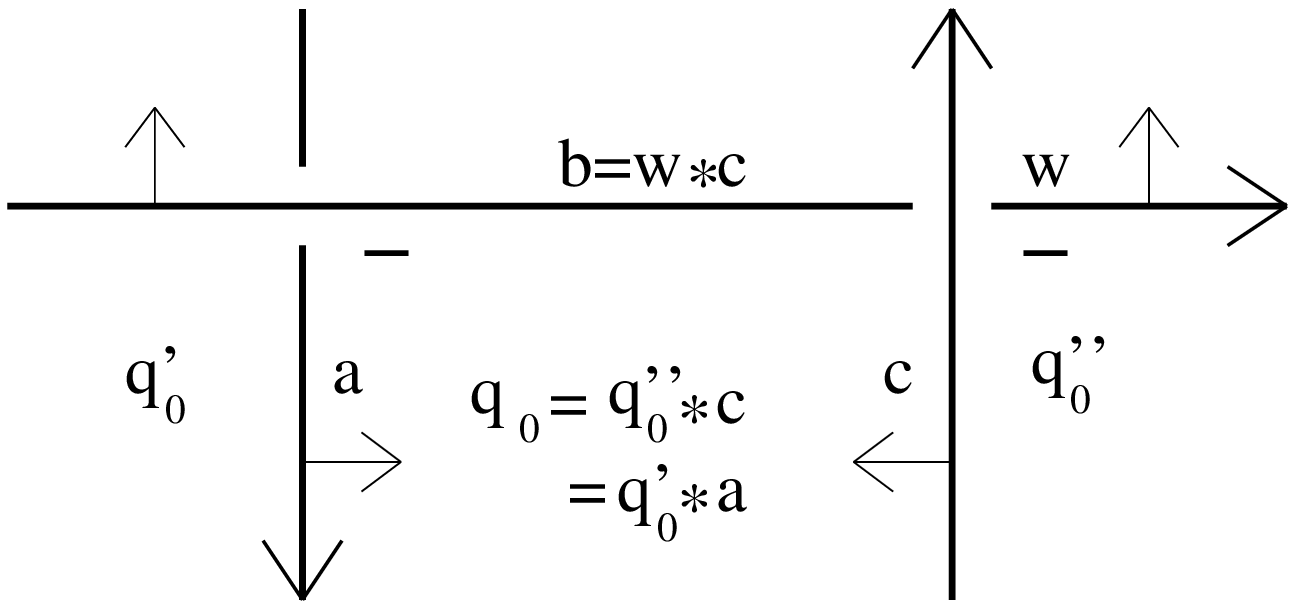,height=2.7cm}\ \ \ \ \ \ \
 \psfig{figure=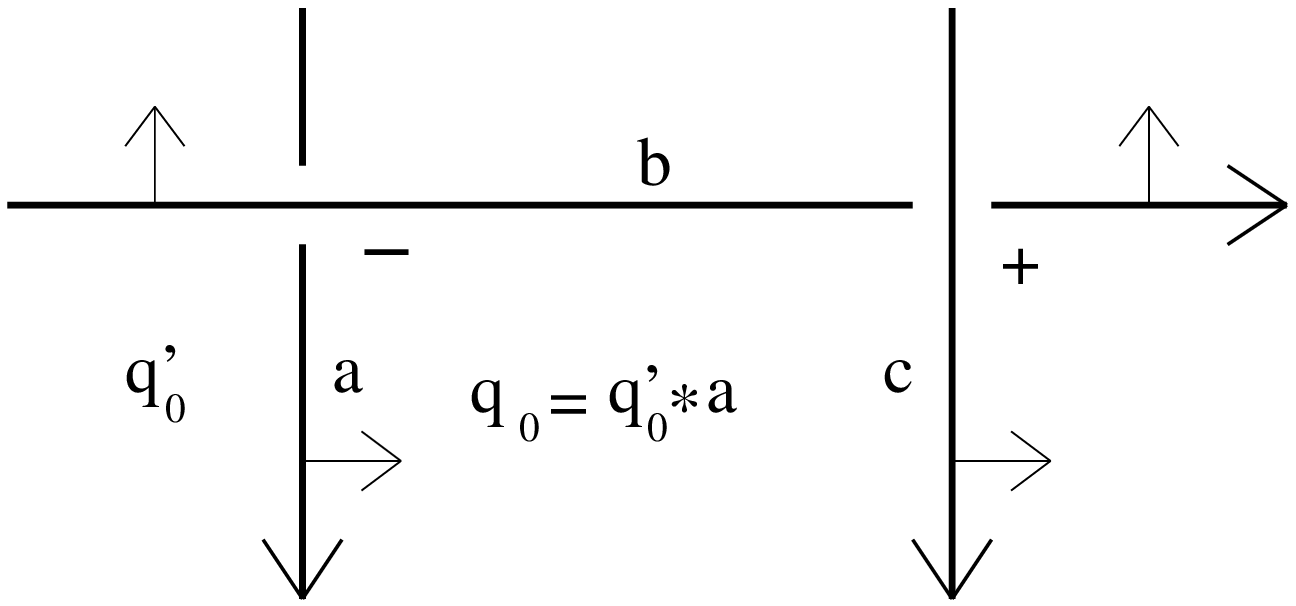,height=2.7cm}} \ \\
\centerline{Figure 2.1; In all case the connecting arc has label $q_{sh}(arc)$ equal to $(q_0,b)$;}
\centerline{ the multi-labelling $q_{sq}$ is explained in full generality in Definition \ref{Definition 2.2}}
\ \\

To describe precisely the outcome of our pictures we denote the shadow 3-chain of our diagram by $c_3(D)$,
the contribution of the first crossing $p_1$ by $c_3(p_1)$, the contribution of the second 
crossing, $p_2$, by $c_3(p_2)$, and $c_3(p_1,p_2)= c_3(p_1)+c_3(p_2)$ denotes the contribution 
of both crossings. We also use the notation of Figure 2.2 which is the special case of 
Definition \ref{Definition 2.2}. \ \\
\centerline{\psfig{figure=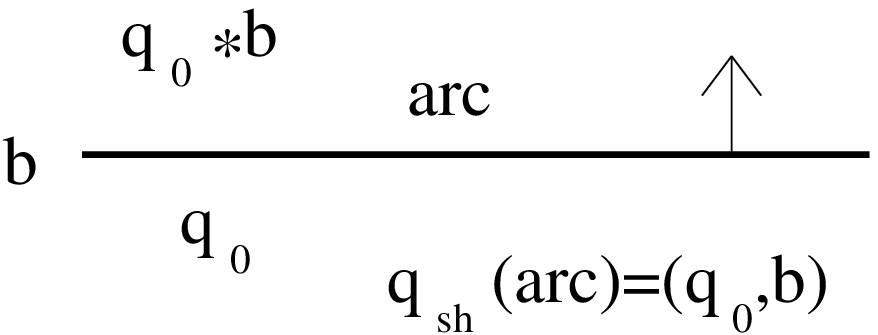,height=2.5cm}} \ \\
\centerline{Figure 2.2; Convention for arc coloring, $q_{sh}(arc)= (q_0,b)$}
\ \\
Thus in the first case we have:
$$c_3(p_1)= (q_0,a,b),\ c_3(p_2)=(q_0,b,c), \mbox{ and } c_3(p_1,p_2)= (q_0,a,b)+ (q_0,b,c).$$
We have then: 
$$d_2^{(*_0)}(c_3(p_1))= d_2^{(*_0)}((q_0,a,b)=d_3^{(*_0)}((q_0,b,c)=(q_0,b)=q_{sh}(arc).$$
Thus proper pieces of $\partial(p_1)$ and $\partial(p_2)$ cancel out in $c_3(p_1,p_2)$. 
This will be the case always, and below we shortly analyze other cases:\\
In the second case we have 
$$c_3(p_1)= (q_0,a,b),\ c_3(p_2)=-(q''_0,w,c), \mbox{ and } c_3(p_1,p_2)= (q_0,a,b)- (q''_0,w,c).$$
Where $b=w*c$. \ We have then:
$$d_2^{(*_0)}((q_0,a,b)=d_3^{(*)}((q''_0,w,c)=(q_0,b)=q_{sh}(arc).$$
In the third case we have:
$$c_3(p_1)= -(q'_0,a,b),\ c_3(p_2)=-(q''_0,w,c), \mbox{ and } c_3(p_1,p_2)= -(q'_0,a,b)- (q''_0,w,c).$$
We have then
$$d_2^{*}((q'_0,a,b)=d_3^{*}((q''_0,w,c)=(q_0,b)=q_{sh}(arc).$$
Finally, in In the fourth case we have:\\
$$c_3(p_1)= -(q'_0,a,b),\ c_3(p_2)=+(q_0,b,c), \mbox{ and } c_3(p_1,p_2)= -(q'_0,a,b) + (q_0,b,c).$$
We have then:
$$d_2^{(*)}((q'_0,a,b)=d_3^{(*_0)}((q_0,b,c)=(q_0,b)=q_{sh}(arc).$$
This proves that $c_3(D)$ is a cycle in $C_3(X)$, as $\partial=\partial^{(*_0)}- \partial^{(*)}=
\sum_{i=2}^3(-1)^id^{(*_0)}_i - \sum_{i=2}^3(-1)^id^{(*)}_i$. 
We didn't consider all cases as we will argue in the general case 
that all cases follows at once, however we illustrate one more case, of type $(2,2)$, that is 
the horizontal line is above both crossings:\
 We have:
$$c_3(p_1)= -(q'_0,a,b),\ c_3(p_2)=-(q_0,c,b), \mbox{ and } c_3(p_1,p_2)= -(q'_0,a,b) - (q_0,c,b).$$
We have then:
$$d_2^{(*)}((q'_0,a,b)=d_2^{(*_0)}((q_0,c,b)=(q_0,b)=q_{sh}(arc); \mbox{ as illustrated in Figure 2.3}.$$
\ \\
\centerline{\psfig{figure=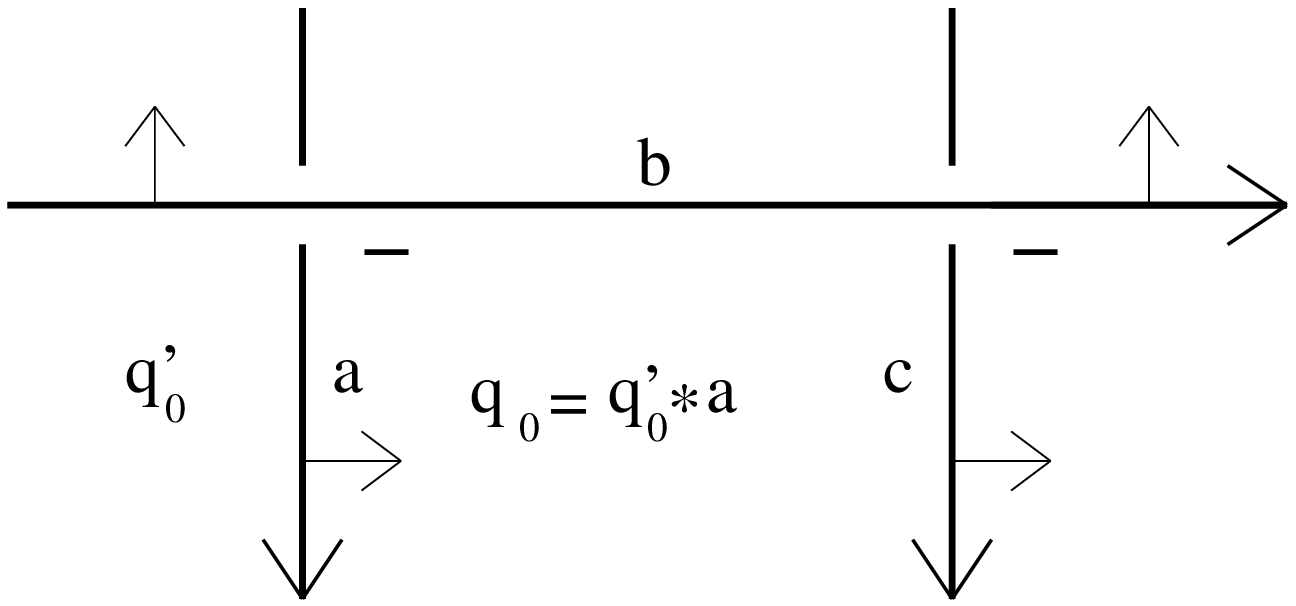,height=3.1cm}} \ \\
\centerline{Figure 2.3; two crossings of type $(2,2)$}
\ \\

We also can work with a tangle diagram $T$ in place of a link diagram $D$; then $c_3(T)$ is not an 
absolute cycle, but we can work in the setting of relative chain (of $(T,\partial T)$). We do not follow 
this idea here but it maybe useful in many situations.


\subsection{The general case of $(n+1)$ and $(n+2)$ cycles}
For the general case we need a notation for coloring of strata of a neighborhood of a crossing of 
multiplicity $(n+1)$ in $\R^{n+1}$ generalizing coloring and shadow coloring.

 For a given vector $w$ in ${\mathbb R}^{n+1}$ let $V_{w}$ be an $(n+1)$-dimensional
linear subspace orthogonal to $w$. For basic vectors $e_1=(1,0,...,0),
\ldots e_i=(0,....0,1,0,...,0), \ldots$
$ e_{n+1}=(0,...,0,1)$ we write
$V_i$ for $V_{e_i}$. We have $\bigcap_{i=1}^{n+1} V_i = (0,...,0)=p$ and it is
our ``model" singularity (crossing of multiplicity $n+1$). In a standard way we associate to this
crossing the signum $+1$. If our system of hypersurfaces $\bigcup V_i$ is a part of a knotting diagram,
the sign $+1$ would correspond to the situation $V_1 >V_2>...>V_{n+1}$, that is $V_i$ above $V_{i+1}$ 
at the crossing. However, our convention, motivated by right self-distributivity, makes more 
convenient assumption $V_1 <V_2<...<V_{n+1}$  and then with our convention the crossing $p$ has the signum 
$sgn(p)=(-1)^{n(n+1)/2}$.

We introduce two labelings (generalizing colorings and shadow colorings): 
\begin{enumerate}
\item [(i)] $q: \bigcup V_i \to X \cup X^2 \cup ... \cup X^{n+1}$;
\item [(ii)] $q_{sh}: {\mathbb R}^{n+1} \to X\cup X^2 \cup ... \cup X^{n+2}$
\end{enumerate}
The strata of the labeling of a point $x$ depends on the ``order of singularity"
that is $q(x)$ (respectively $q_{sh}(x)$) is an element of $X^k$,
(respectively $X^{k+1}$) where $X$ is a fixed shelf and $k=k(x)$ is
the number of hyperplanes $V_i$ to which $x$ belongs (if $x \in {\mathbb R}^{n+1}-\bigcup V_i$ then $k(x)=0$). 
Both colorings are coherent because
of right self-distributivity law in $X$. The idea of coloring is that we choose a color, say $q_0$
for a source region, $R_0$, of ${\mathbb R}^{n+1}-\bigcup_{i=1}^{n+1} V_i$ (this will be part of $q_{sh}$ coloring,
that is $q_{sh}(x)=q_0$ for $x\in R_0$). Furthermore, we choose colors $q_1,q_2,...,q_{n+1}$ and
 if $x\in R_i^{(s)}$ where $R_i^{(s)}$ is the source sheet of $V_i$,
then $q(x)=q_i$ and $q_{sh}(x)=(q_0,q_i)\in X^2$.

With an assumption that $V_i$ is always below $V_{i+1}$ in our considerations. 
  we propagate our colors according to our rules of coloring.
Notice that we use only $*$ (never $\bar *$) in our coloring, so assumption that $X$ is a shelf suffices
here.\\
We describe this idea formally below.
\begin{definition}\label{Definition 2.2}
Choose a shelf $X$ and $n+2$ elements $(q_0,q_1,...,q_n,q_{n+1})$ in $X$.
\begin{enumerate}
\item[(i)] For $x=(x_1,...,x_n,x_{n+1})\in {\mathbb R}^{n+1} - \bigcup V_i$ the
label $q_{sh}(x)$ is defined to be $q_0*q_{i_1}*...*q_{i_s}$
where $0< i_1 < ...<i_s$ are precisely these indexes for which
$x_{i_j} >1$. In particular, if for all $i$, $x_i < 0$ (i.e. $x$ is
in the source sector), then $q_{sh}(x) = q_0$.
\item[(ii)] If $x=(x_1,...,x_n,x_{n+1})$ has only one coordinate, say $i$th,
equal to zero (that is $x \in V_i$ but $x \notin V_j$ for $j\neq i$)
then $q(x)$ is defined to be $q_i*q_{i_1}*...*q_{i_s}$
where $i< i_1 < ...<i_s$ are precisely these indexes for which $i_j >i$ and
$x_{i_j} >1$. In particular, if for all $j$ such that $j>i$ we have $x_j < 0$,
 then $q(x) = q_i$.
\item[(iii)] Let $x=(x_1,...,x_n,x_{n+1})$ belongs to exactly $k$ hyperplanes,
$x \in \bigcap _{j=1}^k V_{i_j}$, then
$q(x)= (q^{(i_1)},...,q^{(i_k)}) \in X^k$, where $q^{i_j}(x)=q(x^{i_j})$
where $x^{i_j}$ is obtained from $x$ by replacing all coordinates
equal to $0$, but $x_{i_j}$, by $-1$ (recall that $x_{i_1}= x_{i_2}=...=
x_{i_k}=0$ and other coordinates are different from $0$).
\item[(iv)] If $x \in \bigcup V_i$ then $q_{sh}(x)$ is obtained from
$q(x)$ by $q_{sh}(x)= (q_{sh}(x'),q(x))$ where $x'$ is obtained from
$x=(x_1,...,x_n)$ by replacing all $0$ in the sequence by $-1$ (i.e. $x'$ is a point in 
a source chamber). In
particular, if $q(x)=q_i$ then $q_{sh}(x)= (q_0,q_i)$.
\end{enumerate}
\end{definition}
 
We can complete with this notation, the proof that $c_{n+1}(D_M)$ and $c_{n+2}(D_M)$ are cycles.


Very schematic visualization of the general case is shown in Figure 2.4 where 
the vertical lines represent $n$ dimensional
sheets of $M^*=f\pi(M^n)$ with source colors $a$ and $c$ respectively and horizontal line representing 
the line of intersection of $n$ sheets (in $R^{n+1}$) with the coloring $b=q(arc)=(q_1,...,q_n)$.
\ \\
\centerline{\psfig{figure=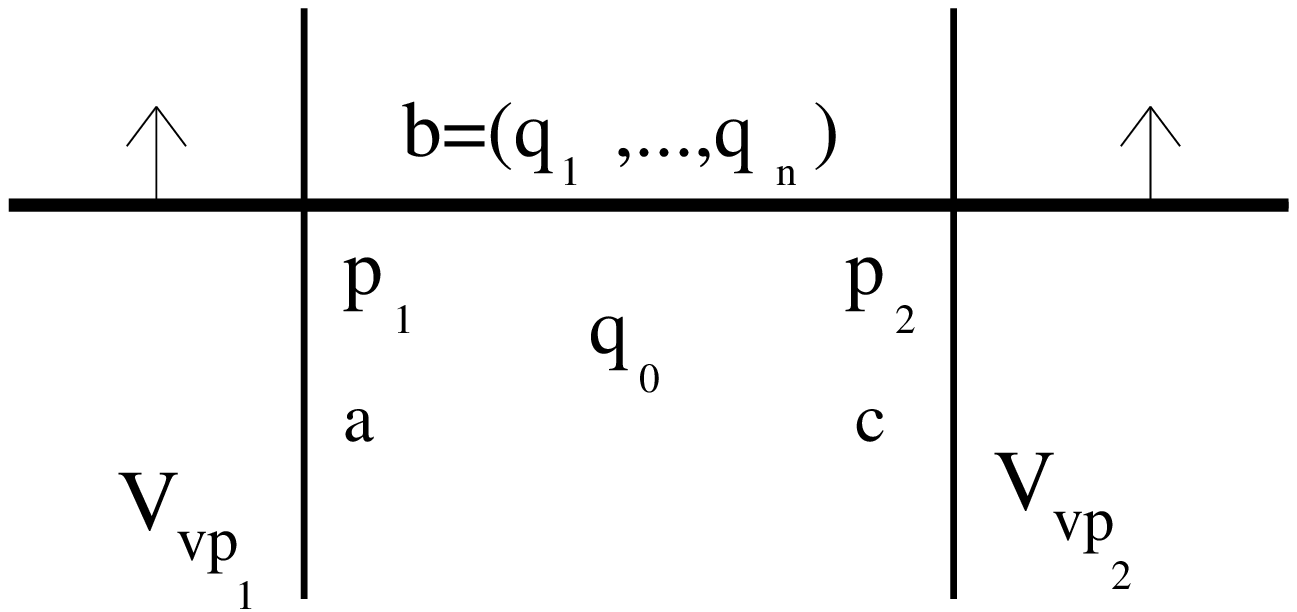,height=3.7cm}}
\centerline{Figure 2.4;  connecting arc has label $q_{sh}(arc)=(q_0,q(arc))=(q_0,b)=(q_0,q_1,...,q_n)$}
\ \\
We denote the position of the first vertical sheet as $i(p_1)$ and the second vertical sheet by $i(p_2)$.
The $n+2$ chains corresponding to $p_1$ and $p_2$ depend on direction of vertical vectors,
$n_{p_1}$ and $n_{p_2}$ to vertical sheets $V_{vp_1}$ and $V_{vp_2}$.
We write $\epsilon(p_1,p_2)= 1$ if the vector $n_{p_1}$ points from $p_1$ to $p_2$ and $0$ otherwise.
Similarly $\epsilon(p_2,p_1)= 1$ if the vector $n_{p_2}$  points from $p_2$ to $p_1$,
and it is $0$ otherwise. This notation is used to identify operation $*_0$ and $*=*_1$. Then we have:
$$d_{i(p_1)}^{(*_{\epsilon(p_1,p_2)})}(q_0,q_1,...,q_{i(p_1)-1},a,q_{i(p_1)},...,q_n) = $$
$$d_{i(p_1)}^{(*_{\epsilon(p_1,p_2)})}(q_0,q_1,...,q_{i(p_1)-1},a,q_{i(p_1)},...,q_n) = $$
$$d_{i(p_2)}^{(*_{\epsilon(p_2,p_1)})}(q_0,q_1,...,q_{i(p_2)-1},b,q_{i(p_2)},...,q_n)$$
Thus $sgn(p_1)(-1)^{i(p_1)}=-sgn(p_2)(-1)^{i(p_2)}$ and in $c_{n+2}(p_1,p_2)$ the terms
$d_{i(p_1)}^{(*_{\epsilon(p_1,p_2)})}$ and $d_{i(p_2)}^{(*_{\epsilon(p_2,p_1)})}$ cancel out.
In conclusion, the $(n+2)$-chain $c_{n+2}(D_M,\tilde\phi)$ is a cycle.

The similar proof works for $c_{n+1}(D_M,\phi)$. The fact that $c_{n+1}(D_M,\phi)$ is a cycle 
follows also from the following observation:
\begin{observation}\label{Observation 2.3}
Consider the map $\gamma_n: C_n(X) \to C_{n-1}(X)$ given by cutting the first coordinate, that is
$\gamma_n(x_1,x_2,...,x_n)= (x_2,...,x_n)$. Then by Lemma \ref{Lemma 1.8} (compare Remark \ref{Remark 1.10}(iv))
 the map $(-1)^n\gamma$ is 
a chain map for $\partial^{(*_0)}$ and $\partial^{(*)}$ thus also for $\partial= \partial^{(*_0)}-\partial^{(*)}$.
From definitions of $c_{n+1}$ and $c_{n+2}$,
 we have $c_{n+1}(D_M,\phi)= \gamma_{n+2}c_{n+2}(D_M,\tilde\phi)$. From this we conclude that if
$c_{n+2}(D_M,\tilde\phi)$ is an $(n+2)$-cycle then $c_{n+1}(D_M,\phi)$ is an $(n+1)$-cycle.
\end{observation}

In the next two sections we  show that our cycles (and their sums) are topological invariants.
We will start from the fact that the spaces of colorings (and shadow colorings) are topological invariants.

\section{Topological invariance of colorings}\label{Section 3}
 
The logic of the section is as follows: 
We introduce here, following \cite{Joy-2,F-R}, two definitions of a fundamental rack or quandle of
a knotting, the abstract one and the concrete definition. 
The abstract definition of the fundamental rack or quandle of 
a knotting is independent on any projection in the similar way, as the fundamental group.
In the concrete definition, for a given projection, we get the concrete presentation of 
a fundamental rack or quandle from the diagram using generators and relations in a way reminiscent of 
the Wirtinger presentation of the fundamental group of a classical link complement.
It was observed in \cite{F-R,FRS-2} that these definitions are equivalent using a general position argument.
As a consequence we get that a concrete rack and quandle colorings are topological invariants 
(independent on a diagram) because in both cases abstract and concrete colorings are obtained 
from a homomorphism from the fundamental object to $X$.

\subsection{The fundamental rack and  quandle of an $n$-knotting}\label{Subsection 3.1}

The first definition, we give, uses a knotting diagram.
We follow Definitions \ref{Definition 1.11} and \ref{Definition 1.13}, except that in place of 
concrete chosen $(X;*)$ we build a universal (called fundamental) object (magma, shelf, rack or quandle).

\begin{definition}(Fundamental Magma of a knotting diagram)\label{Definition 3.1}
Let $f: M  \to {\mathbb R}^{n+2}$ be a knotting,  $\pi: {\mathbb R}^{n+2} \to {\mathbb R}^{n+1}$ a regular 
projection and $D_M$ related knot diagram.
\begin{enumerate}
\item[(i)]
The fundamental magma $X(D_M)=(X;*)$  is given by the following finite presentation.
The generators of $X$ are in bijection with the set of regions of $M$
cut by lower decker set. Relations in $(X;*)$ are given as follows: if $R_1$ and $R_2$ are two regions separated by
$n$-dimensional upper decker region $R_3$ (with variables, respectively, $q_1$, $q_2$ and $q_3$) and 
the orientation normal to $R_3$
 points from $R_1$ to $R_2$, then $q_1*q_3=q_2$.\footnote{Notice that however our definition 
of $X(D_M)$
is not using any properties of $*$ but still for $n>1$ and a projection with a triple crossing point, 
colors involved in the crossing satisfy right self-distributivity (see Figure 1.8 where $(q_1*q_2)*q_3=
(q_1*q_3)*(q_2*q_3)$).}
\item[(ii)] If $(X;*)$ is required to be a shelf we get a fundamental shelf of $D_M$.
\item[(iii)] If $(X;*)$ is required to be a rack we get a fundamental rack of $D_M$. 
\item[(iv)] If $(X;*)$ is required to be a quandle we get a fundamental quandle of $D_M$.
\item[(v)] Fix a quandle $(X;*)$ and the knotting $f: M \to {\mathbb R}^{n+2}$, with regular projection 
$\pi: {\mathbb R}^{n+2} \to {\mathbb R}^{n+1}$. The quandle (resp. rack) coloring of a diagram $D_M$ 
is a quandle homomorphism from the fundamental quandle (resp. rack) $X(D_M)$ to $X$.
\end{enumerate}
\end{definition}
The quandle (or rack) coloring described in Definition \ref{Definition 3.1} is equivalent to 
Definition \ref{Definition 1.11}.

The Fundamental shadow magma, shelf, rack and quandle of a  knotting diagram are defined analogously 
to that of coloring; we give a full definition so it is easy to refer to it.
\begin{definition}(Fundamental shadow magma of a knotting diagram)\label{Definition 3.2}
Let $f: M  \to {\mathbb R}^{n+2}$ be a knotting,  $\pi: {\mathbb R}^{n+2} \to {\mathbb R}^{n+1}$ a regular
projection and $D_M$ related knot diagram.
\begin{enumerate}
\item[(i)]
The fundamental shadow magma $X_{sh}(D_M)=(X;*)$  is given by the following finite presentation.
The generators of $X$ are in bijection with the set $R \cup R_{cha}$ where $R$ is the set
 of regions of $M$ cut by lower decker set and $R_{cha}$ is the set of chambers 
of $\R^{n+1}- \pi f(M)$. 
Relations in $(X;*)$ are given as follows: if $R_1$ and $R_2$ are two regions separated by
$n$-dimensional upper decker region $R_3$ (with variables, respectively, $q_1$, $q_2$ and $q_3$) 
and the orientation normal to $R_3$
 points from $R_1$ to $R_2$, then $q_1*q_2=q_3$.
Furthermore, if   $\tilde R_1$ and $\tilde R_2$ are $n+1$ chambers 
separated by $n$ dimensional region $\alpha$ where the orientation normal
of $\alpha$  points from $\tilde R_1$ to $\tilde R_2$,  and $\tilde q_1,\tilde q_2,\tilde q_3$ are 
colors of $\tilde R_1, \tilde R_2$ and $\alpha$,
respectively, then $\tilde q_1*\tilde q_3=\tilde q_2$.\footnote{Notice that however our definition of 
$X_{sh}(D_M)$ 
is not using any properties of $*$ but still if there is a  double crossing in the projection then
colors involved in the crossing satisfy right self-distributivity (see Figure 1.6). }
\item[(ii)] If $(X;*)$ is required to be a shelf we get a fundamental shelf of $D_M$.
\item[(iii)] If $(X;*)$ is required to be a rack we get a fundamental rack of $D_M$.
\item[(iv)] If $(X;*)$ is required to be a quandle we get a fundamental quandle of $D_M$.
\item[(v)] Fix a quandle (or a rack) $(X;*)$ and the knotting $f: M \to {\mathbb R}^{n+2}$, with regular projection
$\pi: {\mathbb R}^{n+2} \to {\mathbb R}^{n+1}$. The quandle (resp. rack) shadow coloring of a diagram $D_M$
is a quandle homomorphism from the fundamental quandle (resp. rack) $X_{sh}(D_M)$ to $X$.
\end{enumerate}
\end{definition}
Again, the quandle (or rack) shadow coloring described in Definition \ref{Definition 3.2} is equivalent to
Definition \ref{Definition 1.13}.

\begin{remark}\label{Remark 3.3}
If we assume that $X(D_M)$ is the fundamental rack (or quandle) of a diagram $D_M$,
 then the presentation of the fundamental shadow rack $X_{sh}(D_M)$ can be 
obtained by adding one generator and no new relations (except that of rack (or quandle) relations). That is:
$X_{sh}(D_M)= \{X(D_M), w\ | \ \}$. The new generator $w$ is a color of an arbitrary, but fixed, chamber 
of $\R^{n+1}-D_M$. The presentation can be justified using Lemma \ref{Lemma 1.12} or by 
a method described in Footnote 8 to Definition \ref{Definition 1.13}. 
\end{remark} 

We recall in the next subsection a projection free approach to the fundamental rack and quandle of 
a knotting and use it to notice that $X(D_M)$ and $X_{sh}(D_M)$ are independent on the concrete diagram.

\subsection{Abstract definitions of a fundamental rack and quandle}
Joyce, Fen, and Rourke \cite{Joy-2,F-R} gave an abstract definition of the fundamental rack of a knotting, 
independent on 
a projection and they noted that it is equivalent to the concrete definition given in Subsection 3.1.

We follow here \cite{F-R} in full generality, however we are concerned mostly with the case of 
of the ambient manifold $W=\R^{n+2}$.
\begin{enumerate}
\item[(i)]
Let $L: M \to W$ be a knotting (codimension two embedding). We shall assume that the embedding is
 proper at the boundary if $\partial M \neq \emptyset$, that $W$ is connected and that $M$ is transversely 
oriented in $W$. In other words we assume that each normal disk to $M$ in $W$ has an orientation which 
is locally and globally coherent. The link is said to be {\it framed} if there is given cross section 
(called framing) $\lambda: M \to \partial N(M)$ of the normal disk bundle (the total space of the bundle is 
a tubular neighborhood of $L(M)$ in $W$). Denote by $M^+$ the image of $M$ under 
$\lambda$. We call $M^+$ the parallel manifold to $M$.
\item[(ii)] We consider homotopy classes $\Gamma$ of paths in $W_0=\mbox{closure}(W-N(M)$ from a point 
in $M^+$ to a base point. During the homotopy the final point of the path at the base point is kept fixed and the 
initial point is allowed to wander at will in $M^+$.
\item[(iii)] The set $\Gamma$ 
is a right $\pi_1(W_0)$-group-set, that is the fundamental group of a knotting complement acts on $\Gamma$
as follows:
let $\gamma$ be a loop in $W_0$ representing an element $g$ of the fundamental group. 
If $a\in \Gamma$ is represented by the path $\alpha$ then define $a\cdot g$ to be the class of the composition path 
$\alpha \circ \gamma$. We can use this action to define a rack structure on $\Gamma$. Let $p\in M^+$ be 
a point on the framing image. Then $p$ lies on a unique meridian circle of the normal circle bundle. 
Let $m_p$ be the loop based at $p$ which follows round the meridian in a positive direction. 
Let $a,b \in \Gamma$ be represented by the paths $\alpha,\beta$ respectively. Let $\partial(b)$ be 
the element of the fundamental group determined by the loop $\bar\beta\circ m_{\beta}\circ \beta$.
(here $\bar\beta$ represents the reverse path to $\beta$ and $m_{\beta}$ is an abbreviation for $m_{\beta(0)}$ 
the meridian at the initial point of $\beta$.) The {\it fundamental rack} of the framed link $L$ is defined to be
 the set $\Gamma=\Gamma(L)$ of homotopy classes of paths as above with operation
$$a*b= a\cdot \partial(b)=  [\alpha\circ \bar\beta \circ m_{\beta} \circ \beta].$$
\item[(iv)] A rack coloring, by a given rack $(X;*)$ is a rack homomorphism $f: \Gamma(L) \to X$.
In a case of $W=R^{n+2}$ we give also down to earth definition from the link projection (initially depending on 
the projection) (see Definition \ref{Definition 1.11}). 
\item[(v)] If $L$ is an unframed link then we can define its {\it fundamental quandle}:\ 
let $\Gamma_q=\Gamma_q(L)$ be the set of homotopy classes of paths from the boundary of the regular 
neighborhood ($N(M)$) to the base point where the initial point is allowed to wander during the course 
of the homotopy over the whole boundary. The rack structure on $\Gamma_q(L)$ is defined similarly to that 
of $\Gamma(L)$. Thus the fundamental quandle of $L$ is the quotient of the fundamental rack of $L$ by relations 
generated by idempotency $x*x=x$. 
\item[(vi)] A quandle coloring,  by a given quandle $(X;*)$, is a quandle homomorphism $f: \Gamma_q(L) \to X$.
\end{enumerate}

\ \\
\centerline{\psfig{figure=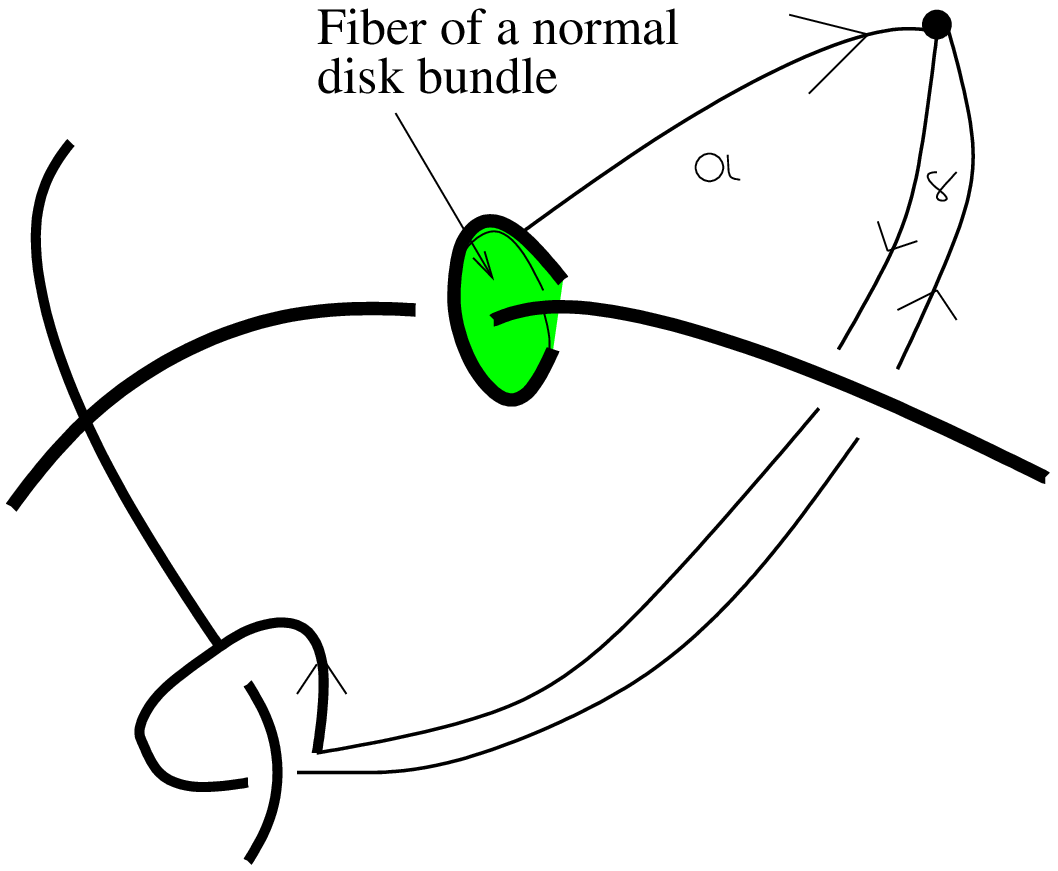,height=3.2cm}}
\centerline{Figure 3.1; Composition $a\cdot g$ where $a$ is a class of an arc 
$\alpha$ and $g$ is a class of a loop $\gamma$}\ \\

For $W=R^{n+2}$ (or $S^{n+2}$) our two definitions of the fundamental rack (or quandle) are equivalent. 
If $D_M$ is a diagram of a knotting $L:M \to \R^{n+1}$ with the regular projection $\pi: \R^{n+2}\to \R^{n+1}$ 
then we have a natural epimorphism $F_q: X_q(D_M) \to \Gamma_q(L)$ given as follow: 
Let $x_H$ be a generator of $X_q(D_M)$ corresponding to a sheet (region) $H$ of $D_M$. Choose a base point 
of $\R^{n+1}- \pi L(M)$ very high (call it $\infty)$ and project it to a point of $H$ cutting $\partial V_M$ 
at some point $m_H$, then we define $F_q(x_H)$ to be the class of a straight line from $m_H$ to $\infty$. 
Similarly we define a rack epimorphism $F: X(D_M) \to \Gamma(L)$, by extending $F_q(x_H)$ by starting from the 
point of $M^+$ being on the same fiber disk of $V_M$ as $m_H$ and connecting along the boundary of the disk to $m_H$.
\begin{theorem}\label{Theorem 3.4}(\cite{F-R,FRS-2})
\begin{enumerate}
\item[(i)]
Two definitions of a fundamental rack (resp. quandle) of a (framed) $n$-link $L:M \to \R^{n+2}$  
coincide, the map $F:X(D_M) \to \Gamma(L)$ is a quandle isomorphism.  
In particular, Definition \ref{Definition 3.1} (and equivalent \ref{Definition 1.11}) for racks 
and quandles are independent on regular projection\footnote{In quandle case we consider knottings up to 
(smooth) ambient isotopy (equivalently up to Roseman moves), 
and in a rack case up to framed (smooth) ambient isotopy.} and give a
finite presentation of the fundamental rack $\Gamma(L)$ (resp. quandle $\Gamma_q(L)$). 
\item[(ii)]
The fundamental shadow rack (resp. quandle) is independent on regular projection thus it is well 
defined for a knotting $n$-link $L:M \to R^{n+2}$; we denote it by $\Gamma^sh(L)$ (resp. $\Gamma_q^sh(L)$). 
\item[(iii)] The sets ($X$-quandle-sets) $Col_X(D_M)$ a $Col_{X,sh}(D_M)$ 
do not depend on the diagram of a given linking $M$.
\end{enumerate}
\end{theorem}
\begin{proof} The statement and a sketch of a proof is given in \cite{F-R} (Remarks(2) p. 375\footnote{Fenn and 
Rourke write in Remarks(2): {\it A similar analysis can be carried out for an embedding of $M^n$ in $S^{n+2}$: 
we obtain a ``diagram" by projecting onto ${\mathbb R}^{n+1}$ in general position and regarding 
top dimensional strata ($n$-dimensional sheets) as ``arcs" to be labelled by generators and 
$(n-1)$-dimensional strata (simple double manifolds) as ``crossings" to be labelled by relators. In general 
position a homotopy between paths only crosses the $(n-1)$-strata and a proof along the lines of the 
theorem can be given that this determines a finite presentation of the fundamental rack.}}) and  
\cite{FRS-2}(Lemma 3.4; p.718). One can give also a proof using Roseman moves (starting with the pass move $S(c,n+2,0)$
discussed in detail later in this paper).

We also can conclude that the set $X_{sh}(D_M)$ (in fact, $X$ quandle set) of shadow colorings is a topological invariant.
We use here  Remark \ref{Remark 3.3}, namely, as $X(D_M)$ is a topological invariant, so is
$X_{sh}(D_M)= \{X(D_M), w\ | \ \}$.

\end{proof}

\begin{remark}\label{Remark 3.5}
Theorem \ref{Theorem 3.4} should be understood as follows: If $R$ is a Roseman move on a diagram of $n$-knotting
$D_M$ resulting in $RD_M$, then there are natural $X$-rack isomorphisms (i.e. bijections preserving right action by $X$), 
$R_{\#}: Col_X(D_M) \to Col_X(RD_M)$ and $\tilde R_{\#}: Col_{X,sh}(D_M) \to Col_{X,sh}(RD_M)$. Natural means here 
that outside a ball (tangle) in which the move $R$ takes place, the bijections $R_{\#}$ and $\tilde R_{\#}$ are 
identity. \ This raises an interesting question: assume that after using a finite number of Roseman moves 
we come back to the diagram $D_M$.  What automorphism of $X$-quandle-sets $Col_X(D_M)$ and $Col_{X,sh}(D_M)$ we performed?
Is it always an inner automorphism\footnote{If we change a base point in the definition of the 
fundamental rack we make an inner automorphism on it (i.e. generated by $X$ action) reflecting similarity 
with fundamental group), \cite{Joy-2}.}. When it is the identity?
\end{remark}

\begin{corollary}\label{Corollary 3.6}
The homology, $H_*^W(X(D_M))$ and $H_*^W(X_{sh}(D_M))$ of the fundamental rack and the fundamental shadow rack 
are topological knotting invariants.
\ M.Eisermann proved that in the classical case a knot is nontrivial if and only if $H_2^Q(X(K))=\Z$, \cite{Eis-1}.
We can ask what we can say in a general case about $H_{n+1}^Q(X(D_M))$.
\end{corollary}

The presentation of $X(D_M)$ gives the coloring of $D_M$ by the quandle  $X(D_M)$; we call this 
the fundamental coloring and denote by $\phi_{fund}$. Similarly, The presentation of $X_{sh}(D_M)$ gives the 
shadow coloring of $D_M$ by the quandle  $X_{sh}(D_M)$; we call this     
the fundamental shadow coloring  and denote by $\tilde \phi_{fund}(D_M)$.

\begin{corollary}\label{Corollary 3.7}
The homology classes of cycles $c_{n+1}(D_M,\phi_{fund})$ (resp. $c_{n+2}(D_M,\tilde \phi_{fund})$ 
are knotting invariants up to isomorphism of homology groups generated by an automorphism of 
a fundamental quandle (resp. fundamental shadow quandle). 
M.Eisermann proved that in the classical case, the homology class of the fundamental cycle\footnote{This 
class is called in \cite{Eis-1} the {\it orientation class} of $K$.} 
of a nontrivial knot is a generator of $H_2^Q(X(K))=\Z$, \cite{Eis-1}.
  \end{corollary}

\section{Roseman moves are preserving homology classes of fundamental cycles}\label{Section 4}

We deal in this section with the main result of our paper, about (co)cycle invariants of knottings.
 We start by describing precisely the case of a pass 
move, $R$, (generalization of the third Reidemeister move), that is a move of type $S(c,n+2,0)$ 
in notation of \cite{Ros-1,Ros-2}; see Section \ref{Section 6} (we write $R\in S(c,n+2,0)$).

\subsection{Colorings and homology under Roseman moves}\label{Subsection 4.1}

Fix a quandle $(X;*)$. We already established bijection, for any Roseman move $R$ 
between sets of colorings of $D_M$ and the set of colorings of $RD_M$, Theorem \ref{Theorem 3.4}. 
We denote this bijection 
by $R_{\#}$, so $R_{\#}: Col_X(D_M) \to Col_X(RD_M)$. In fact,  $R_{\#}$ is $X$-quandle-sets isomorphism,
 that is it preserves the right multiplication by elements of $X$ ( i.e. $R_{\#}(\phi*x) = R_{\#}(\phi)*x$). .

\begin{theorem}\label{Theorem 4.1} 
For a fixed $\phi \in Col_X(D_M)$ the cycles $c_{n+1}(D_M,\phi)$ and 
$c_{n+1}(RD_M,R_{\#}(\phi))$ are homologous in $H^Q_{n+1}(X)$. Similarly for $\tilde\phi\in Col_{X,sh}(D_M)$,
 the cycles $c_{n+2}(D_M,\tilde\phi)$ and 
$c_{n+2}(RD_M,R_{\#}(\tilde\phi))$ are homologous in $H^Q_{n+2}(X)$.
\end{theorem}
The main, and, as we see later, essentially the only one nontrivial to check is the pass move $R$ of 
type $S(c,n+2,0)$.

\begin{lemma}\label{Lemma 4.2}
The pass move $R\in S(c,n+2,0)$ preserves the homology class of $c_{n+1}(L)$ and $c_{n+2}(L)$. 
To be precise let $\phi$ be a fixed coloring of $D_M$ and $R_{\#}(\phi)$ the corresponding coloring of $RD_M$ 
then the cycles $c_{n+1}(D_M,\phi)$ and 
$c_{n+1}(RD_M,R_{\#}(\phi))$ are homologous in $H^Q_{n+1}(X)$.
Similarly for $\tilde\phi\in Col_{sh,X}(D_M)$,
 the cycles $c_{n+2}(D_M,\tilde\phi)$ and
$c_{n+2}(RD_M,R_{\#}(\tilde\phi))$ are homologous in $H^Q_{n+2}(X)$.
\end{lemma}
\begin{proof} We show the result for all possible types of pass moves (including all possible
co-orientation of sheets) at once. We start from $n+2$ sheets (hypersurfaces) in $\R^{n+2}$ with arbitrary
co-orientation, intersecting in a point $p$, and the direction of time $\vec t$ in general position 
to co-orientation vectors.\\
Before we give technical details, we first use a simple visualization of our proof: \\
On each side of the move, say for $t=-1$ and $t=1$ we have $n+2$ crossings $(p_1,...,p_{n+2})$ 
and $(p'_1,...,p'_{n+2})$ respectively.
Each crossing represents the intersection of $n+1$ sheets in $\R^{n+1}$ that is $p_i$ is the intersection 
of all sheets $V_1,...,V_{n+2}$ but $V_i$ at $t=-1$ and $p'_i$ is the intersection 
of all sheets but $V_i$ at $t=1$. 
 Furthermore, we have $sgn(p_i)=sgn(p'_i)$ for any $i \geq 1$. We concentrate on the case of 
a shadow coloring $\tilde\phi$ of $D_M$ (the non shadow case 
being similar). The weights associated to $p_i$ and
$p'_i$ are $sgn (p_i)d_i^{(*_0)}q_{sh}(p)$ and $sgn (p'_i)d_i^{(*)}q_{sh}(p)$ (the order depends on
co-orientation $\vec n_i$ of $V_i$). Furthermore, the sign of $p_i$ is $(-1)^{n-i}sgn (\vec n_i \cdot \vec t)$ 
that is the sign depend on whether $\vec n_i$ agrees or disagrees with $\vec t$ (that is the scalar product 
$\vec n_i \cdot \vec t$ is positive or negative).
In effect $c_{n+2}(p_i)- c_{n+2}(p'_i)= \epsilon (-1)^{-i}(d_i^{(*_0)}- d_i^{(*)})(q_{sh}(p)$, where 
$\epsilon = \pm 1$,  and in 
effect $c_{n+2}(D_M)-c_{n+2}(RD_M)= \pm \partial_{n+3}(q_{sh}(p))$.
Therefore 
$c_{n+2}(D_M)-c_{n+2}(RD_M)$ is homologous to zero in $H_{n+2}(X)$ as needed\footnote{We were informed 
by Scott Carter that this
observation was crucial in the definition of rack homology by
Fenn, Rourke and Sanderson. In particular, the relation of the generalized
Reidemeister move can be read from the boundary of singularity of 
one dimension higher. We deal then with a point $\hat p$ of multiplicity $n+2$ and we choose any time 
vector $\vec t$ in general position to normal vectors of $n+1$-dimensional hyperplanes. 
We shadow color neighborhood of $\hat p$ so that $q_{sh}(\hat p)= (q_0,q_1,...,q_{n+2})$. Then we analyze 
face maps $d_i^{*_0}(q_{sh}(\hat p))$ and $d_i^{*}(q_{sh}(\hat p))$ and recognize $q_{sh}$ of points 
$p_1$,...,$p_{n+2}$ and $p'_1,...,p'_{n+2}$ at crossection at $t=-1$ and $t=1$. Figure 4.1 illustrate 
it for $n=1$.}.
Similarly $c_{n+1}(D_M,\phi)-c_{n+1}(RD_M,R_{\#}(\phi))= \pm \partial_{n+2}(q(p))$. This follows also
directly by using Observation \ref{Observation 2.3}.
\end{proof}
\ \\
\centerline{\psfig{figure=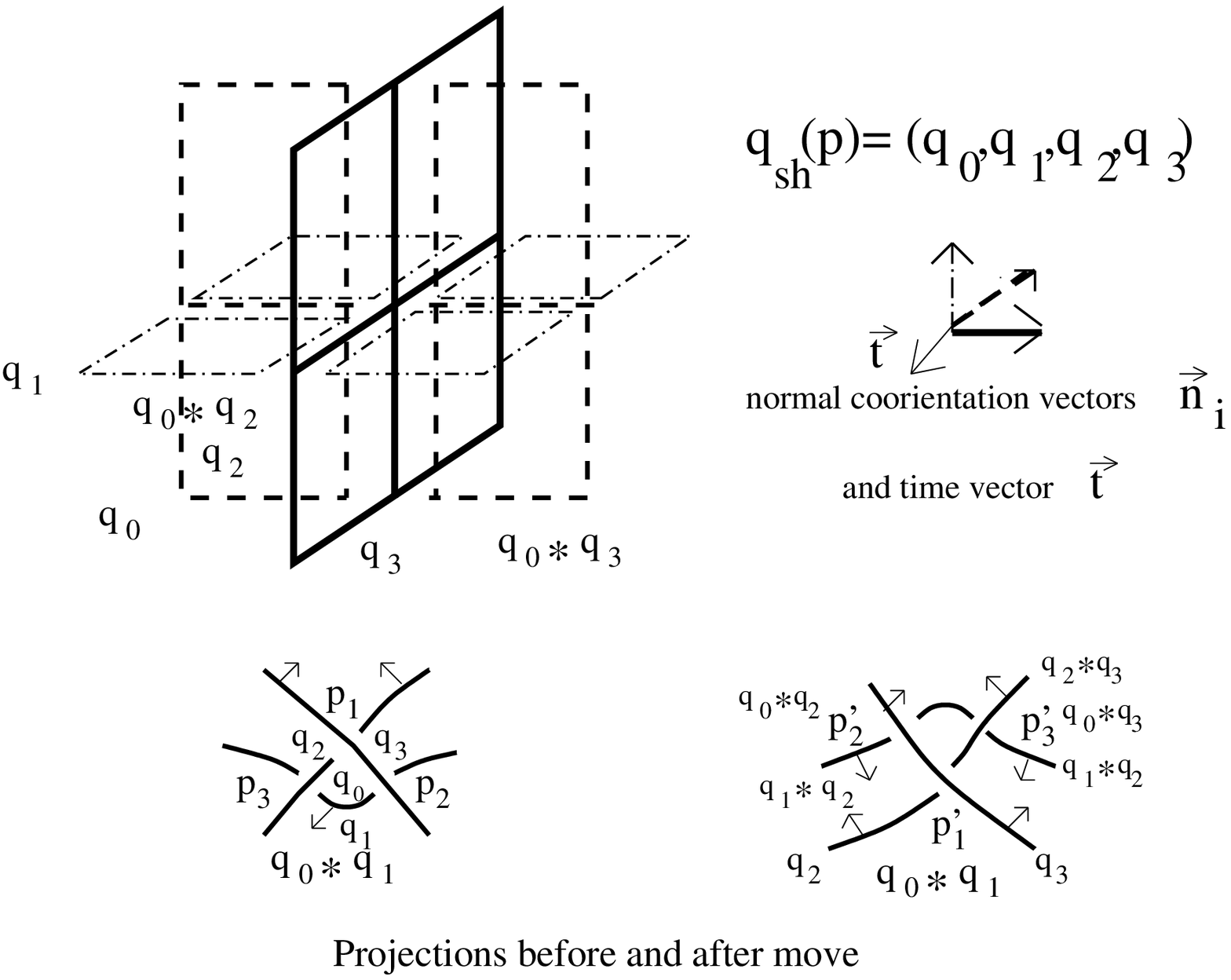,height=10.1cm}}
\ \ \\
\centerline{Figure 4.1;  From isotopy to pass move ($n=1$ case)}
\centerline{\small {$c_3(p_1,p_2,p_3)= (q_0,q_2,q_3)- (q_0,q_1,q_3)+ (q_0,q_1,q_2)=\partial^{(*_0)}(q_0,q_1,q_2,q_3)$}}
\centerline{\small {$c_3(p'_1,p'_2,p'_3)= (q_0*q_1,q_2,q_3)- (q_0*q_2,q_1*q_2,q_3)+ (q_0*q_3,q_1*q_3,q_2*q_3)
=\partial^{(*)}(q_0,q_1,q_2,q_3)$}}

\ \\

Now we can complete the proof of Theorem \ref{Theorem 4.1}:\\
 We use Roseman moves in more substantial way then before.  Details of Roseman theory is given in 
Section \ref{Section 6} were we follow \cite{Ros-1,Ros-2,Ros-3}. 
Here we give a short description referring often to that section.
  D.~Roseman proved that for any $n$ there is a finite number of moves
on link diagrams in ${\mathbb R}^{n+1}$ so that if two diagrams $F_1^{n}$ and
$F_2^{n}$  represent ambient isotopic links in ${\mathbb R}^{n+2}$ then our
diagrams are related by a finite number of Roseman moves. For $n=1,2$ and $3$,
the moves of Roseman were explicitly given (for $n=1$ these are classical Reidemeister moves).

We are showing that any Roseman move is preserving the homology class of $c_{n+1}(D_M)$ and $c_{n+2}(D_M)$.
Because only crossings of multiplicity $n+1$  are contributing to cycles $c_{n+1}(D_M)$ and $c_{n+2}(D_M)$,
thus is suffices to consider only those Roseman moves which involve singularities of multiplicity 
$n+1$ before or after the move. A precise definition of Roseman moves and their properties is 
given in Section \ref{Section 6} and here we need only the fact that there are exactly three types of moves of interest:
\begin{enumerate}
\item[(i)] A move of type $S(c,n+2,0)$ which we analyzed in Lemma \ref{Lemma 4.2} called the pass move or 
maximal crossing move or the generalized third Reidemeister move.
\item[(ii)] A move of type $S(c,n+1,0)$  (or its inverse a move of type $S(c,n+1,1))$.
This move describes two
cancelling crossing points of a knotting diagram and can be called the
generalized second Reidemeister move (in the case of $n=2,3$ they are moves $(e)$ in \cite{Ros-1}).\\
In the isotopy the arc of points of multiplicity $n+1$
 joins these crossing points, and they have opposite signs. Furthermore, up to sign,
this crossings have the same contributions to $c_{n+1}(D_M,\phi)$ (and $c_{n+2}(D_M,\tilde\phi)$). 
Thus in the state sum of Definitions \ref{Definition 1.16} and \ref{Definition 1.17} they do cancel.\\
Notice that we can interpret our situation as a special case of considerations in Subsection 2.2
(consider Figure 2.4 with $V_{vp_1}=-V_{vp_2}$ and $a=c$ as describing the knotting diagram before the move).

\item[(iii)] A move of type $S(m,(1,n-1),0,p)$ (with $p=0$ or $1$), where one side
of the isotopy has a point
of multiplicity $n+1$ (compare the move (f) in the case of $M^2$ in $R^4$ and the move ($\ell$)
in the case of $M^3$ in $R^5$ \cite{Ros-1}).\footnote{There is a misprint in \cite{Ros-1} page 353; it should be
$S(m,(1,1),0,0)$ or $S(m,(1,1),0,1)$ in place of $S(m,(1,2),0,0)$ or $S(m,(1,2),0,1)$.}
Then
the branch set $B$ is the boundary of the lower decker set $D_-$ so the lower decker set does not separate
the regions; thus both sides of this set have the same color. 
Therefore the chains corresponding to the  crossing of multiplicity  $n+1$ are degenerate in $C^R_{n+1}(X)$ 
and $C^R_{n+2}(X)$, thus these chains do not contribute to quandle homology $H^Q_{n+1}(X)$ and 
$H^Q_{n+2}(X)$, respectively.
\end{enumerate}
This complete our proof of Theorem \ref{Theorem 4.1}

If $M= M_1 \cup M_2 \cup... \cup M_k$ we can generalize Theorem \ref{Theorem 4.1} for 
a non-shadow coloring of $D_M$ (we use notation of Definition \ref{Definition 1.18}).
Our proof of Theorem \ref{Theorem 4.1} also work in this case:
\begin{corollary}\label{Corollary 4.3}
For a fixed $\phi \in Col_X(D_M)$  and a Roseman move $R$ the difference of chains before and after 
the move,  $c_{n+1}(D_M,\phi,i)) - c_{n+1}(RD_M,R_{\#}(\phi,i))$ is a boundary (so homologically trivial).
\end{corollary}

\subsection{Cycle invariants of knottings}

To obtain invariants of knottings using Theorem \ref{Theorem 4.1} we can either sum over 
all colorings of the cycles $c_{n+1}(D_M,\phi)$ or take them as a set with multiplicity (in order not 
to loose an information that some colorings have the same cycle):
\begin{theorem}\label{Theorem 4.4} 
Let $(X;*)$ be a fixed quandle,  $f: M \to \R^{n+2}$ be an $n$-knotting,
$\pi:\R^{n+2}\to \R^{n+1}$ a regular projection, and $D_M$ the knotting diagram. We use the notation 
$[c]$ for a homology class of a cycle $c$.
\begin{enumerate}
\item[(1)] Let $[c_{n+1}(D_M),\phi]$ denotes the homology class of the cycle $c_{n+1}(D_M,\phi)$. 
For a finite $X$ the state sum, defined below   
$$[c_{n+1}(D_M)]= \sum_{\phi\in Col_X(D_M,\phi)} [c_{n+1}(D_M),\phi]$$ in 
the group ring  $ZH_{n+1}(X)$ is a topological invariant of a knotting $M$. Thus we can denote this 
invariant by  $c_{n+1}(M)$  and call it the (non-shadow) cycle invariant of a knotting $f: M \to \R^{n+2}$ 
(or shortly of $M$).
\item[(2)] The reduced (non-shadow) cycle invariant of the knotting $f: M \to \R^{n+2}$ 
$$c^{red}_{n+1}(M)= [c^{red}_{n+1}(D_M)]= \sum_{\phi\in Col_{red,X}(D_M)} [c_{n+1}(D_M,\phi)]$$
is a
topological invariant. Notation is explained as follows. 
We sum here over smaller number of colorings  using the fact that set of colorings $Col_X(D_M)$ 
is an $X$-quandle-set and as proven in Observation  \ref{Observation 1.5} 
$c_{n+1}(D_M),\phi)$ is homologous to $c_{n+1}(D_M),\phi*x)$, for any $x\in X$. Thus we take 
$Col_{red,X}(D_M)$ to be the subset of all $X$ coloring, one coloring from every orbit. Even if we can 
have various choices for $Col_{red,X}(D_M)$ the resulting $c^{red}_{n+1}(M)= [c^{red}_{n+1}(D_M)]$ is 
well defined. 

\item[(3)] Let $[c_{n+2}(D_M),\tilde \phi]$ denotes the homology class of the cycle $c_{n+2}(D_M,\tilde\phi)$. 
For a finite $X$ the state sum  $$[c_{n+2}(D_M)]= \sum_{\tilde\phi\in Col_{sh,X}(D_M)} [c_{n+2}(D_M,\tilde\phi)]$$ in 
the group ring  $ZH_{n+2}(X)$ is a topological invariant of a knotting $M$. Thus we denote this  
invariant by  $c_{n+2}(M)$  and call it the shadow cycle invariant of a knotting $M$.

\item[(4)] The reduced shadow cycle invariant of the knotting $M$ is a 
topological invariant  
$$c^{red}_{n+2}(M)= [c^{red}_{n+2}(D_M)]= \sum_{\tilde\phi\in Col_{red,sh,X}(D_M)} [c_{n+2}(D_M,\tilde\phi)].$$

\item[(5)] If in any sum of (1)-(4) we replace sum by a set with multiplicity, we obtain topological invariants,
$c^{set}_{n+1}(M)$, $c^{red,set}_{n+1}(M)$, 
$c^{set}_{n+2}(M)$, and  $c^{red,set}_{n+2}(M)$
respectively (we allow $X$ to be infinite here).

\end{enumerate}
\end{theorem}

\subsection{Cocycle invariants of knottings}
In this subsection we reformulate our main result in the language of cocycles and cohomology.
For a fixed quandle $(X;*)$ and fixed  cocycles in $C^{n+1}(X,A)$ and $C^{n+2}(X,A)$ we 
obtain directly cocycle invariants of $n$-knotting. It generalizes the case of $n=1,2$ (see \cite{CKS-3} for a summary),
and the case of $n=3$ checked in \cite{Rosi-2}. We start from the definition which 
involves diagrams.

\begin{definition}\label{Definition 4.5} Let $(X;*)$ be a fixed quandle,  $f: M \to \R^{n+2}$ an $n$-knotting,
$\pi:\R^{n+2}\to \R^{n+1}$ a regular projection, and $D_M$ the knotting diagram.
\begin{enumerate}
\item[(1)]  For a fixed coloring $\phi\in Col_X(D_M)$ and $(n+1)$-cocycle $\Phi: {\Z}X^{n+1} \to A$
we define the value $\Phi(D_M,\phi)\in A$ by
$$\Phi(D_M,\phi) = \Phi(c_{n+1}(D_M,\phi))=\sum_p \Phi(c_{n+1}(p,\phi)),$$
where $\Phi(c_{n+1}(p,\phi))$ is a Boltzmann weight of the crossing $p$ of multiplicity $n+1$, and the sum 
is taken over all crossings of $D_M$.
\item[(2)]  We can also take into account the fact that $M$ is  not necessary connected 
(following \cite{CENS,CKS-3} in the case $n=1$).
That is if $M=M_1\cup M_2\cup...\cup M_k$,
we can take the sum 
from  (1)  not over all crossings $p$ of multiplicity $n+1$ but only those which
have $M_i$ on the bottom of the crossings. Let us denote such a set of crossings by ${\mathcal T}_i$. Then
we define
$$\Phi(D_M,\phi, i)= \sum_{p\in {\mathcal T}_i}  \Phi(c_{n+1}(p,\phi))$$
\item[(3)]
For a fixed shadow coloring $\tilde\phi\in Col_{X,sh}(D(M))$ and $(n+2)$-cocycle $\tilde\Phi: {\Z}X^{n+1} \to A$,
we define the value $\tilde\Phi(D_M,\tilde\phi)\in A$ by the formula:
$$\tilde\Phi(D_M,\tilde\phi) = \tilde\Phi(c_{n+2}(D_M,\tilde\phi))=\sum_p \tilde\Phi(c_{n+2}(p,\tilde\phi)).$$
\end{enumerate}
\end{definition}
\begin{theorem}(Cocycle invariants)\label{Theorem 4.6} Consider a knotting $f:M \to \R^{n+2}$ and
fix a quandle $(X;*)$ and $(n+1)$-cocycle $\Phi: {\Z}X^{n+1} \to A$ and  
$(n+2)$-cocycle $\tilde\Phi: {\Z}X^{n+2} \to A$.
\begin{enumerate}
\item[(1)] For a fixed colorings $\phi \in Col_X(D_M)$,
the element $\Phi_X(D_M,\phi)\in A$ is  preserved by any  Roseman 
move $R$ that is $\Phi_X(D_M,\phi)=\Phi_X(RD_M,R_{\#}(\phi))$ in $A$.
\item[(2)] If $M=M_1\cup M_2\cup...\cup M_k$ then the conclusion of (1) holds also for $\Phi(D_M,\phi, i)$, 
that is $\Phi_X(D_M,\phi,i)=\Phi_X(RD_M,R_{\#}(\phi),i)$ in $A$.
\item[(3)]  For a fixed shadow coloring  $\tilde\phi \in Col_{X,sh}(D_M)$, the element
$\tilde\Phi_X(D_M,\tilde\phi)\in A$ is preserved by any  Roseman                      
move, that is 
$\tilde\Phi_X(D_M,\tilde\phi)= \tilde\Phi_X(RD_M,R_{\#}(\tilde\phi))$ in $A$.
\end{enumerate}
\end{theorem}

\begin{proof} Theorem \ref{Theorem 4.6} follows directly from the analogous result for homology.
We should stress that we do not need here the property that $c_{n+1}(D_M,\phi,i)$  are cycles, as we 
can evaluate a cocycle on any chain. Furthermore, we can work also with tangles not only with knotting diagrams.
\end{proof}

To produce invariant of a knotting we should make our invariants of diagram independent on 
a choice of a coloring (and use Theorems \ref{Theorem 4.4} and \ref{Theorem 4.6}).
As before, two natural solutions are to take a set with multiplicity of 
invariants over all colorings or, for finite $X$, sum invariants over all colorings as usually 
is done in statistical mechanics.

\begin{definition}\label{Definition 4.7} Let $(X;*)$ be a fixed quandle,  $f: M \to \R^{n+2}$ be an $n$-knotting, 
$\pi:\R^{n+2}\to \R^{n+1}$ a regular projection, and $D_M$ the knotting diagram.
\begin{enumerate} 
\item[(1)]  Let $\Phi$ be a fixed cocycle  in $C_Q^{n+1}(X)$. Then:
\begin{enumerate}
\item[(i)] $$\Phi_X(D_M) = \sum_{\phi\in Col_X(D_M)} \Phi_X(D_M,\phi) = \sum_{\phi} \prod_p \Phi(c_{n+1}(p,\phi)),$$
where $X$ is a finite quandle. Here we have classical cocycle invariant in $ZA$, written as a state sum 
($A$ n a multiplication notation) and 
generalizing cocycle invariants of \cite{CKS-3}.
\item[(ii)]
$$\Phi_X(D_M, i)= \sum_{\phi\in Col_X(D_M)} \Phi_X(D_M,\phi, i),$$
where $M=M_1\cup M_2\cup...\cup M_k$ and we consider only the crossings
 which have $M_i$ on the bottom of the crossings.
For $n=1,2$ this invariant of $M$ with ordered components was described in \cite{CENS,CKS-3}.

\item[(iii)]
$$\Phi_X^{red}(D_M) = \sum_{\phi\in Col_{X,red}(D_M)} \Phi_X(D_M,\phi).$$
Here we use the fact that $X$ acts on $Col_X(D_M)$ and we can choose in the sum one 
representative from any orbit. Any choice is good (see Observation \ref{Observation 1.5})
 and we write for the chosen subset $Col_{X,red}(D_M)$,\footnote{For example, 
for the trivial knotting $S^n \subset \R^{n+2}$, finite $X$ and any cocycle, we have 
$\Phi_X(S^2)= |X|\cdot 1$ and $\Phi_X^{red}(S^2)= |{\mathcal O}_r|\cdot 1$, where 
${\mathcal O}_r$ is the set of orbit of the action of $X$ on $X$ on the right. Furthermore,
 in our notation $1$ is a zero of an abelian group written 
multiplicatively, and $|{\mathcal O}_r|\cdot 1$ is an element of a group ring $Z(ZX^{n+1})$.}

\item[(iv)] $$\Phi_X^{red}(D_M,i) = \sum_{\phi\in Col_{X,red}(D_M)} \Phi_X(D_M,\phi,i);$$
here we reduce crossings as in (ii) and colorings as in (iii).
\item[(v)] Without restriction to finite $X$ we can repeat all definitions of (i)-(iv) by 
considering, in place of the sum over  colorings, the set of invariants indexed by colorings 
(thus we have a set with multiplicities, or better cardinalities of elements if $X$ is infinite).
We get $\Phi_{set}(D_M)$, $\Phi_{set}(D_M,i)$, $\Phi^{red}_{set}(D_M)$, and $\Phi^{red}_{set}(D_M,i)$, 
respectively.
\end{enumerate}
\item[(2)] For a fixed $(n+2)$-cocycle $\tilde\Phi: {\Z}X^{n+2} \to A$, 
we define:
\begin{enumerate} 
\item[(i)]
$$\tilde\Phi_X(D_M) = \sum_{\tilde\phi\in Col_{X,sh}(D_M)} \tilde\Phi_X(D_M,\tilde\phi) = \sum_{\tilde\phi}
 \prod_p \tilde\Phi(c_{n+2}(p,\tilde\phi)),$$ 
where $X$ is a finite quandle. Here we have classical shadow cocycle invariant in $ZA$
\item[(ii)]
$$\tilde\Phi_{X,red}(D_M) = \sum_{\tilde\phi\in \tilde Col_{X,red}(D_M)} \tilde\Phi_{X,red}(D_M,\tilde\phi).$$
\item[(iii)] Without restriction on $X$ to be finite, we can repeat definitions of (i) and (ii) by
considering, in place of the sum over  colorings, the set with multiplicity of invariants indexed by colorings.
We get $\tilde \Phi_{set}(D_M)$ and $\tilde \Phi_{set,red}(D_M)$.
\end{enumerate}
\end{enumerate}
\end{definition}

\begin{theorem}(Cocycle invariants)\label{Theorem 4.8} Consider a knotting $f:M \to \R^{n+2}$ and 
for a fixed quandle $(X;*)$,  quandle cocycles $\Phi: {\Z}X^{n+1} \to A$ and
 $\tilde\Phi: {\Z}X^{n+2} \to A$. Then:

\begin{enumerate}                                                                         
\item[(1)] If $X$ is finite then $\Phi_X (M)$  $\Phi_X^{red}(M)$, $\tilde\Phi_X (M)$, and  $\tilde\Phi_X^{red}(M)$
 are topological invariants of the knotting 
(i.e. independent on a diagram,  invariant under Roseman moves). They are called cocycle, and shadow cocycle 
invariants of a knotting $M$.
\item[(2)] For any $X$, $\Phi_{set}(M)$, $\Phi_{set}^{red}(M)$,  $\tilde\Phi_{set} (M)$, and $\tilde\Phi_{set}^{red}(M)$
 are topological invariants of the knotting $M$.
\item[(3)] If $M=M_1\cup...\cup M_n$ and $X$ is finite then  $\Phi_X (M,i)$ and $\Phi_X^{red}(M,i)$ are
 topological invariants of the knotting $M$ with ordered components. Similarly, for any $X$ the sets with multiplicity 
$\Phi_X^{set} (M,i)$ and $\Phi_{X,red}^{set}(M,i)$ are  topological invariants of the knotting $M$ with ordered components.
\item[(4)] We can make invariants of (3) to be independent on the order of components if we take the set with
multiplicity of invariants over all $i$.
\end{enumerate}
\end{theorem}
\begin{proof} It follows directly from Theorems \ref{Theorem 4.4} and \ref{Theorem 4.6}. 
Notice here that for shadow coloring the idea of considering $M=M_1\cup M_2\cup...\cup M_k$
and only crossings where $M_i$ is on the bottom will not work as $d_2^{(*_0)}$ usually differs from $d_2^{(*)}$, 
thus $c_{n+2}^{red}(p_i) - c_{n+2}^{red}(p'_i)$ is not necessary homological to zero..
In the non-shadow case we only needed $d_1^{(*_0)}=d_1^{(*)}$. 
\end{proof}

\section{Twisted (co)cycle invariants of knottings}\label{Section 5}
Twisted homology (and cohomology) was introduced in \cite{CENS}.
Most of the results of the paper generalize, without much changes to the twisted case so 
we give a concise explanation. 

\begin{definition}\label{Definition 5.1}
\begin{enumerate}
\item[(i)] The twisted chain complex of a shelf $(X;*)$ is given by the chain modules
$C^T_n(X)=Z[t^{\pm 1}]X^n$ (that is a free modules with basis $X^n$ and with coefficients in a ring of Laurent polynomials
in variable $t$), and the chain map $\partial^T= t\partial^{*_0}- \partial^{*}$. Recall that:
$$\partial^{*_0}(x_1,...,x_n)=\sum_{i=1}^n(-1)^i(x_1,...,x_{i-1},x_{i+1},...,x_n), \mbox{ and}$$
$$\partial^{*_0}(x_1,...,x_n)=\sum_{i=1}^n(-1)^i(x_1*x_i,...,x_{i-1}*x_i,x_{i+1},...,x_n).$$
\item[(ii)] If $(X;*)$ is a spindle (e.g. a quandle) we define as in  the untwisted case the degenerate  
and quandle homology. Thus as before we consider $H^{TW}_n(X)$ for $W=R,D$ and $Q$.
\item[(iii)] The cohomology $H^n_{TW}(X,A)$ are defined in a standard way with $A$ being an $Z[t^{\pm 1}]$-module.
\end{enumerate}
\end{definition}

The theory of cocycle invariants, for n=1 or 2, was introduced in \cite{CES-1} for $n=1,2$.
We give definition for any $n$-knotting below.
Our description follow \cite{CKS-3}, the important tool we use is the classical Alexander numbering
of chambers in $(\R^{n+1},\pi f(M))$ (see \cite{CKS-0,CKS-3}). Our version of the definition refers to shadow
colorings by an (extended) shift rack structure on integers with infinity $(\Z\cup \infty;*_s)$ where $a*_sb=a+1$ 
(in particular $\infty *b=\infty$).
\begin{definition} \label{Definition 5.2}
\begin{enumerate}
\item[(i)] Let $X$ be a set and $f:X \to X$ a bijection with a fixed point $b$. We define a rack $(X;*_f)$ by $a*_fb=f(a)$.
Then for a given knotting diagram $D_M$ the shadow rack coloring of chambers of the knottings,
is called the generalized Alexander numbering. More precisely, we color regions of the diagram trivially by $b$,
choose one chamber and color it by an  element of $X-b$ and the resulting shadow coloring of chambers is a 
generalized Alexander numbering.
\item[(ii)] The Alexander coloring of Chambers (e.g. \cite{CKS-3}) starts from the rack 
$(\Z\cup \infty;*_s)$ and the unbounded chamber is colored by $0$.
\end{enumerate}
\end{definition}

\begin{definition}(Twisted chains of knotting)\label{Definition 5.3} 
Let $f: M \to \R^{n+2}$ be an $n$-knotting, $\pi:\R^{n+2}\to \R^{n+1}$ a regular projection, and $D_M$ the
knotting diagram. Furthermore, fix a rack or quandle $X$ , a coloring $\phi: {\mathcal R}\to X$, and a
shadow coloring $\tilde\phi$.
\begin{enumerate}
\item[(1)] If $p$ is a crossing of multiplicity $(n+1)$ then we define the chain (twisted Boltzmann weight) 
associated to $p$ as $c_{n+1}^T(p,\phi)= t^{-k(R_0)}c_{n+1}(p,\phi)$, where $k(R_0)$ is the Alexander 
numbering of the source region in the neighborhood of $p$ and $c_{n+1}(p,\phi)$ is the untwisted 
Boltzmann weight.
\item[(2)] In the case $\tilde\phi$ is the shadow coloring we define a twisted shadow Boltzmann weight by:
$c_{n+2}^T(p,\tilde\phi)= t^{-k(R_0)}c_{n+2}(p,\tilde\phi)$, where $c_{n+2}(p,\tilde\phi)$ is the untwisted 
shadow Boltzmann weight associated to $p$.
\item[(3)] The twisted chain associated to the diagram $D_M$ is
$$c_{n+1}^T(D_M,\phi)=\sum_{p\in Crossings}c_{n+1}^T(p,\phi).$$
\item[(4)] Finally, we sum over all $X$ colorings of $D_M$ so the result is in the
group ring over $C^T_{n+1}(X)$ (in fact, it is in the group ring of $H^T_{n+1}(X)$) 
 It is convenient here to use multiplicative notation for chains so that
$c^T_{n+1}(D_M,\phi)=(\Pi_p(q_1,...,q_{n+1})^{sgn p})^{t^{-k(R_0)}}$ and then
$$c^T_{n+1}(D_M)= \sum_{\phi}c^T_{n+1}(D_M,\phi).$$
\item[(5)] We define the twisted shadow chain associated to the diagram $D_M$ in an analogous manner:
$$c_{n+2}^T(D_M,\tilde\phi)=\sum_{p\in crossings}c_{n+2}^T(p,\tilde\phi).$$ 
Then we sum over all colorings to get:
$$c^T_{n+2}(D_M)= \sum_{\tilde\phi}c^T_{n+2}(D_M,\tilde\phi).$$
\item[(6)] As in untwisted version we can consider smaller sum by taking into account only one element from 
each orbit of action by $X$ on the space of colorings. However we should be careful here about
 which action we consider
because 
the action $(x_1,...,x_n)*x$ is equal on homology to $t\cdot Id$ according to Observation \ref{Observation 1.5}
Thus we should change this 
action to $(x_1,...,x_n) \to t^{-1}(x_1,...,x_n)*x$. We obtain then reduced versions of (4) and (5). 
\item[(7)] Each of the above has its cocycle version as long as we choose a twisted $(n+1)-$ and 
$(n+2)$-cocycles in $C_T^{n+1}(X)$ and $C_T^{n+2}(X)$, respectively. 
\end{enumerate}
\end{definition}
Most of the results as in Theorems \ref{Theorem 4.1}, \ref{Theorem 4.6}, and \ref{Theorem 4.8} 
generalizes without any problem to twisted (co)homology.
We give two examples below.

\begin{theorem}\label{Theorem 5.4}
For a fixed $\phi \in Col_X(D_M)$ the chain $c^T_{n+1}(D_M,\phi)$ is a cycle and it is homologous to 
$c^T_{n+1}(RD_M,R_{\#}(\phi))$ in $H^{TQ}_{n+1}(X)$, where $R$ is any Roseman move on a diagram $D_M$. 
Similarly for a fixed shadow coloring  $\tilde\phi\in Col_{X,sh}(D_M)$,
 the chain $c^T_{n+2}(D_M,\tilde\phi)$ is a cycle and it is homologous to
$c^T_{n+2}(RD_M,R_{\#}(\tilde\phi))$  in $H^{TQ}_{n+2}(X)$.
\end{theorem}
The main, nontrivial Roseman move to check is the pass move $R$ of
type $S(c,n+2,0)$. Here analysis is very similar to that of Theorem \ref{Theorem 2.1} and Lemma \ref{Lemma 4.2}.

\begin{theorem}(Twisted cocycle invariants)\label{Theorem 5.5} 
Consider a knotting $f:M \to \R^{n+2}$ and
fix a quandle $(X;*)$ and $(n+1)$-twisted cocycle $\Phi^T: \Z [t^{\pm 1}]X^{n+1} \to A$ 
and  $(n+2)$-cocycle $\tilde\Phi^T: {\Z [t^{\pm 1}]}X^{n+2} \to A$ where $A$ is a $[t^{\pm 1}]$-module.
\begin{enumerate}
\item[(1)] For a fixed colorings $\phi \in Col_X(D_M)$,
the element $\Phi_X(D_M,\phi)= \Phi^T(c^T_{n+1}(D_M,\phi)\in A$ is  preserved by any  Roseman
move $R$ that is $\Phi_X(D_M,\phi)=\Phi_X(RD_M,R_{\#}(\phi))$ in $A$.
\item[(2)]  For a fixed shadow coloring  $\tilde\phi \in Col_{X,sh}(D_M)$, the element
$\tilde\Phi_X(D_M,\tilde\phi) = \tilde\Phi^T(c^T_{n+2}(D_M.\tilde\phi) \in A$ is preserved by any  Roseman
move, that is
$\tilde\Phi_X(D_M,\tilde\phi)= \tilde\Phi_X(RD_M,R_{\#}(\tilde\phi))$ in $A$.
\item[(2)] We can sum now over coloring of a finite quandle $X$, or sum over reduced colorings, or just take a set
over coloring, to get twisted cocycle invariants of a knotting.
\end{enumerate}
\end{theorem}

\begin{remark}\label{Remark 5.6}
If we work with racks and rack (or degenerate) homology we cannot ignore degenerate elements, 
so the Roseman move of type  $S(m,(1,n-1),0,p)$ (generalized first Reidemeister move) 
cannot be performed on the diagrams without possibly changing 
(co)homology class of (co)cycles. For other Roseman moves however all our results work well.
Thus we have (co)cycle invariants of diagrams of $n$-knottings 
up to all Roseman moves except moves of type  $S(m,(1,n-1),0,p)$,
for any rack.
\end{remark}

\section{General position and Roseman moves in codimension 2}\label{Section 6}
An important tool our work is given by the work of Roseman on 
general position of isotopy on moves in co-dimension two and the moves 
he developed. The next subsections follow \cite{Ros-1,Ros-2,Ros-3}. We have used these notion in the paper; here is 
more formal development.

Before we can define Roseman moves we need several definitions.
 
\subsection{General position}\label{Subsection 6.1}\ 
Let $M=M^n$ be a closed smooth $n$-dimensional manifold and 
 $f: M \to {\mathbb R}^{n+2}$ its smooth embedding which is called a smooth knotting. Define
$\pi: {\mathbb R}^{n+2}\to {\mathbb R}^{n+1}$ given by $\pi(x_1,....,x_{n+1},x_{n+2})= (x_1,....,x_{n+1})$ to be a 
projection on the first $n+1$ coordinates.
 The projection of the knotting is the set
$M^*= \pi f (M)$. Crossing set $D^*$ of the knotting, is the closure in $M^*$ of the set
of all points $x^*\in M^*$ such that $(\pi f )^{-1}(x^*)$ contains at least two points.\\
We define the double point set $D$ as
  $D= (\pi f )^{-1}(D^*)$. The branch set $B$ of $f$ is the set of all points $x\in M$ such that
$\pi f$ is not an immersion at $x$. In general, if $A\subset M$ then $A^*$ denote $\pi f(A)$.

\begin{definition}\label{Definition 6.1} 
Let $f: M \to {\mathbb R}^{n+2}$ be a smooth knotting with branch set $B$ and double point set $D$.
We say that $f$ is in general position with respect to the projection $\pi$ if the following six conditions
hold:\\
(1) $B$ is a closed $n-2$ dimensional submanifold of $M$.\\
(2) $D$ is a union of immersed closed $(n-1)$-dimensional
submanifolds of  $M^n$   with normal crossings. Denote the set of
points of $D$ where normal crossings occur as $N$ and call this
the self-crossing set of $D$.\\
(3) $B$  is a submanifold of  $D$  and for any  $b_0 \in B$ there
is a small  $(n-1)$-dimensional open sub-disk  $V$  with $b_0 \in
V$, $ V \subseteq  D$  such that  $V-B$  has two components $V_0$
and  $V_1$ ,  each of which is an  $(n-1)$-disk which is embedded
by the restriction of   $\pi \circ   f$  but with $V_0^* = V_1^*$ 
(Figure 2.1).\\
(4)  $B$  meets $N$  transversely.\\
(5) $(\pi \circ    f)|B$  is an immersion of  $B$
      with normal crossings.\\
(6)  The crossing set of  $B^*$   is transverse to the crossing
set of $(D-B)^*$. 
\end{definition}

\centerline{\psfig{figure=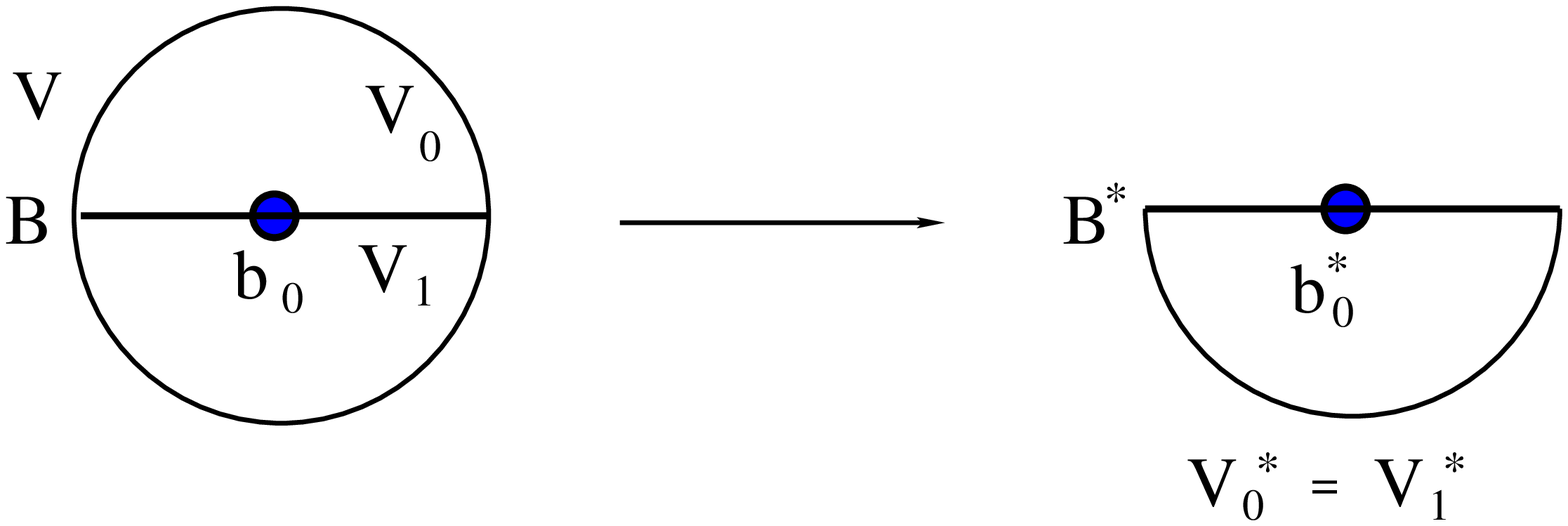,height=5.0cm}}
\begin{center}
 Figure 6.1; projecting (folding) of $V=V_0\cup  V_1\cup (B\cap V)$ onto $(V-B)^*=V_0^*=V_1^*$  
\end{center}

\begin{theorem}\label{Theorem 6.2}  Given a knotting
$f \colon M^n \to {\mathbb R}^{n+2} $      we may isotope  $f$  to a
map which is in general position with respect to the projection $\pi$. 
\end{theorem}

Similarly we define what it means for an isotopy  $F: M\times I \to {\mathbb R}^{n+2}\times I$  to be in
general position with respect to the projection $\pi^\prime =\pi\times Id$ . It
is just the previous definition for general position of a
codimension two knotting except that  $B$  and  $D$  may have
nonempty boundary. In particular, $F_0=F/(M\times \{0\})$, $F_1=F/(M\times \{1\}): M\to {\mathbb R}^{n+2}$ are 
smooth knottings in general position.

\subsection{Arranging for moves} 
We put on our isotopy additional conditions called arranging for moves \cite{Ros-3}.
Roughly speaking, we filtrate $D^*$ of $F: M\times I \to {\mathbb R}^{n+2}\times I$ in such a way 
that the projection $p: {\mathbb R}^{n+1}\times I \to I$ restricted to any component of each stratum, $Q^{(i)}$, 
is a Morse style function. 
\begin{definition}[Roseman]
Let $q$ be a proper immersion of a manifold $Q$ in ${\mathbb R}^{n+1}\times I$. We say that $q(Q)$ is 
immersed in Morse style if $pq$ is a Morse function, where $p: {\mathbb R}^{n+1}\times I \to I$. 
We assume that a Morse function has critical points on different levels. 
\end{definition}

For details see \cite{Ros-3}. Here we just mention that $Q^{(0)}$ is
the crossing set of $B^*$, $Q^{(1)}=B^*$, the projection of the branch set of $F$.
$Q^{(2)}=D^*$, generally, $Q^{(k)}$, $k>1$ is the closure of the subset of $D^*$ such that $F'=\pi'F$ is at least 
$k$ to $1$. Roseman proves:
\begin{theorem}[Roseman]\label{Theorem Ros}
 Any isotopy $F: M\times I \to {\mathbb R}^{n+2}\times I$ can be arranged for moves.
\end{theorem}
 
\subsection{Listing of moves after Roseman}
The standard set of moves $\mathcal{M}^n$ is described as follows:\\
Fix a dimension  $n$  and
suppose we are given an isotopy  $F \colon M^n \times I \to {\bf
{\mathbb R}^{n+2}} \times I$ which is arranged for moves.  This gives a
sequence of elementary singularities. Each singularity will
correspond to a standard local knot move in our collection
$\mathcal{M}^n$.

In the notation which follows, we consider three general types of
points:

\begin{enumerate}

\item \label{btype} \textbf{branch type:} critical points of  $B^*$
and self-crossing points of  $B^*$   for which we use the letter
$b$

\item  \label{ctype} \textbf{crossing type:} critical points
of  $D^*$   and the crossing set of  $D^*$   which do {\em not} belong
to  $B^*$ for which we use the letter  $c$.

\item   \label{mtype} \textbf{mixed type:} critical points which are
in the crossing set of    $D^*$   {\em and} are in    $B^*$,  a
``mixed'' type for which we use the letter  $m$.

\end{enumerate}
  The first collection of branch type points is denoted
$\{S(b,k,p,q)\}$.  If  $x^* \in D^*$ is such a singular point,
where  $D^*$ is the crossing set of an isotopy  $F \colon M^n \times I \to
{\mathbb R}^{n+2} \times I$,  let  $k$  denote the number of points of
$F'^{-1} (x^* )$.  In our case the branch point set $B^*$ of $F$, 
is codimension 2 in $M^n \times I$
that is it is of 
dimension $n-1$. \\

If in projection this branch set intersects
itself generically, the self-intersection set will have dimension
$n-4$.  It follows that  $ 1\leq  k \leq n-2$. The integer $p$ is
the index of the singularity. The integer $q$ has range  $0 \leq q
\leq k$ and might be called {\it transverse index} of this
critical point. This is defined as follows. If $x \in B$ consider
a curve $\delta$ in $D$ transverse to $B$ (recall that $B$ has
codimension one in $D$) so that $\delta^*$ is, except for the
point $x$, the two-to-one image of $\delta$. In the $I$ direction,
the image of this curve has a local maximum or a local minimum at
$b$. Now suppose $b^*$ is a $k$-fold point of $B^*$ then we have
$k$ such curves to consider. The number  $q$ is the number of
those curves for which we have a local maximum. Of course, it
follows that $k-q$ of the curves have a local minimum.

 The next collection of crossing type singularities  is denoted by
$\{S(c,k,p)\}$.  If  $x^*$   is such a singularity,  $k$ denotes
the cardinality of  $F'^{-1}  (x^*)$.    Thus  $k$ is an integer $2
\leq  k  \leq n+2$.  Furthermore, on this set of points, where $F'$
is $k$-to-one,  $x^*$  is a critical point in the $I$ direction,
of index  $p$. A single point has index $0$ by convention.

     Finally  $S(m,(i,j),p,q)$   denotes a mixed
singularities.  Such a singularity $x^*$
has   $F'^{-1}  (x^*)$
consisting of  $i+j$  points, where exactly  $i$  of these points
are in  $B$.  Again  $p$  is the index of the singularity and  $q$
 is an  integer, $ 0 \leq  q \leq  i$  which is the number of
local maxima we get by looking at those  $i$  arcs  transverse to
$B$ at the points of   $F'^{-1}  (x^* ) \cap B$.

\section{A knotting $M^n \stackrel{f}{\rightarrow} F^{n+1}\times [0,1] 
\stackrel{\pi}{\rightarrow} F^{n+1}$}\label{Section 7}

The Roseman (local) moves can be used for any $n$-knotting $f: M^n \to W^{n+2}$ by the following 
classical PL-topology result following from Theorem 6.2 in \cite{Hud} (we will use it in a smooth 
case which can be derived using Whitehead results on triangulation of smooth manifolds).

\begin{lemma}\label{Lemma 7.1}
If $C$ is a compact subset of a manifold
$W$ and $F:W\times I \to W$ is the isotopy of $W$ then there is another
isotopy $\hat F:W\times I \to W$ such that\\
 $F_0 ={\hat F}_0$,
$F_1/C = {\hat F}_1/C$ and there exists a number $N$ such that
the set\\ $\{x\in W \ |\ \hat F/\{x\}\times (k/N,(k+1)/N) \ is \
not \ constant\}$
sits in a ball embedded in $W$.
\end{lemma}

Let $f: M^n \to F^{n+1}\bar\times [0,1]$ be an $n$-knotting where $F^{n+1}$ is an $(n+1)$-dimensional 
manifold and $ F^{n+1}\bar\times [0,1]$ is an $[0,1]$-bundle over $F^{n+1}$ (trivial bundle if $F^{n+1}$ is oriented 
 and the twisted $[0,1]$-bundle over $F^{n+1}$ if $F^{n+1}$ is unorientable. In both cases the manifold is oriented). 
Let $\pi: F^{n+1}\bar\times [0,1] \to F^{n+1}$. 
By Lemma \ref{Lemma 7.1} an embedding $f$ can be assumed to be in general position with respect to $\pi$ 
and every ambient isotopy of a knotting can be decomposed into Roseman moves (on $D_M$).
If $\pi_1(F^{n+1})=0$ then essentially all results of the paper can be also proven for the 
knotting (we need $W=F^{n+1}\bar\times [0,1]$ to be simple connected in Lemma \ref{Lemma 1.12}, 
Remark \ref{Remark 3.3}, and Theorem \ref{Theorem 3.4}). 

\begin{remark}\label{Remark 7.2}
If we do not assume that $\pi_1(F^{n+1})=0$ in the case of $W^{n+2}=  F^{n+1}\times [0,1]$, 
we can still develop the theory of (co)cycle invariants by following \cite{FRS-2} where the 
notion of a reduced fundamental rack is developed (essentially one kills the action of $\pi_1(F^{n+1})$).
Then the reduced fundamental rack (or quandle) is, according to Corollary 3.5 of \cite{FRS-2}, the same as
the fundamental rack (or quandle) obtained by rack (or quandle) abstract coloring of any diagram of the knotting. 
\end{remark}

\section{Speculation on Yang-Baxter homology and invariants of knottings}\label{Section 8}

Yang-Baxter operator  can be thought as a direct generalization of right self-distributivity 
when we go from the category of sets to the category of $k$-modules.

We follow here \cite{Leb-1,Leb-2,Prz-1,Prz-2} describing the classical case $n=1$.

First we note how to get Yang-Baxter operator from a right self-distributive binary operation.
Let $(X;*)$ be a shelf and $kX$ be a free module over a commutative ring $k$ with basis $X$ (we can
call $kX$ a {\it linear shelf}). Let $V=kX$, then $V\otimes V = k(V^2)$ and the operation $*$ yields
a linear map $Y=Y_{(X;*)}: V\otimes V \to V\otimes V$ given by $Y(a,b)=(b,a*b)$. Right self-distributivity
of $*$ yields the equation of linear maps $V\otimes V \otimes V \to V\otimes V\otimes V$ as follows:
$$ \mbox{(1) }\ \ (Y\otimes Id)(Id \otimes Y)(Y\otimes Id) = (Id \otimes Y)(Y\otimes Id)(Id \otimes Y).$$
In general, the equation of type (1) is called a Yang-Baxter equation and the map $Y$ a Yang-Baxter operator.
We also often require that $Y$ is invertible. For example if $Y$ is given by invertible $*$, 
then $Y_{(X;*)}$ is invertible with $Y^{-1}_{(X;*)}(a,b)= (b\bar * a, a)$.

In our case $Y_{(X;*)}$ permutes the base $X\times X$ of $V\otimes V$, so it is called a permutation or
a set theoretical Yang-Baxter operator. Our distributive  homology, in particular our rack homology
 $(C_n,\partial^Y=\partial^{(*_0)}-\partial^{(*)})$ can be thought of as the homology of $Y$.
It was generalized from the Yang-Baxter operator coming from a self-distributive $*$ to any
set theoretical Yang-Baxter operator (coming from biracks or biquandles), \cite{CES-2}. For a general
Yang-Baxter operator, there is no general homology theory (however, compare \cite{Eis-1,Eis-2}).
The goal/hope is to define homology for any Yang-Baxter operator  and develop the homological 
invariants of $n$-knottings (it is done for $n=1$ and a set theoretical Yang-Baxter equation in \cite{CES-2}).
The simple visualization of the distributive face map $d_i^{(*)}$ from Figure 8.1, observed by 
I.Dynnikov during Przytycki's talks in Moscow in May 2012 (and slightly earlier by Victoria Lebed when 
she was writing her PhD thesis \cite{Leb-1}), easily gives a hint to homology of set theoretical Yang-Baxter homology,
and, partially, to general Yang-Baxter homology (this is studied in \cite{Prz-2}).
The homology invariants of $n$-knotting should follow, and combining the method of this paper with 
\cite{Leb-2,Prz-2} looks rather promising.
\\ \ \\
\centerline{\psfig{figure=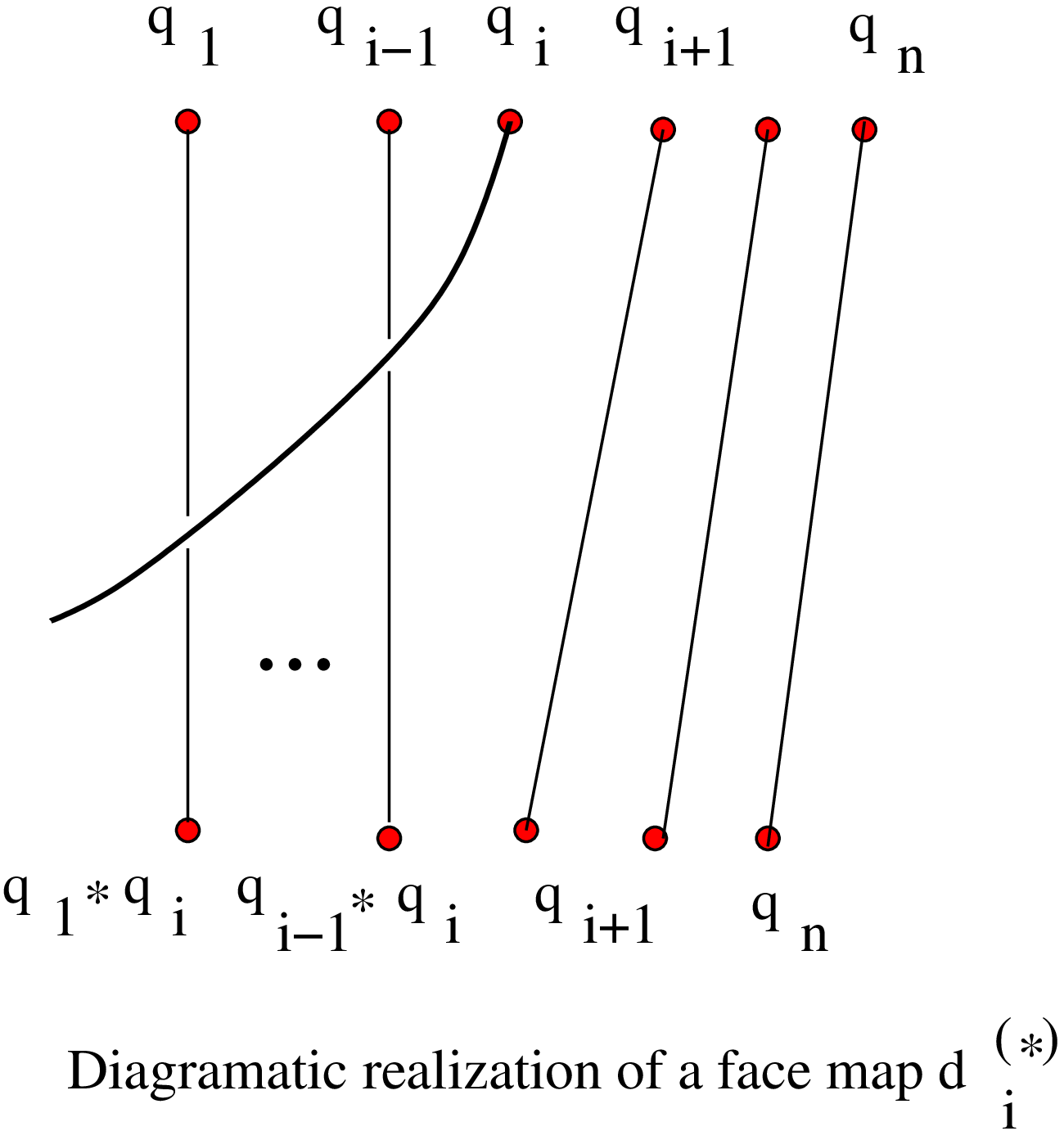,height=9.1cm}}
\\ \ \\
\centerline{Figure 8.1;  Diagrammatic visualization of  a face map gives hint to Yang-Baxter homology.}
\centerline{ For a right self-distributive $*$ we have  a face map }
\centerline{ $d_i^{(*)}(q_1,...,q_n)=(q_1*q_i,...,q_{i-1}*q_i,q_{i+1},...,q_n)$}. 
\centerline{We can also interpret the picture to be applicable to Yang-Baxter theory }
\centerline{by using Yang-Baxter operator at each crossing}

 

\section{Acknowledgments}\label{Section 9}
J.~H.~Przytycki was partially supported by the  NSA-AMS 091111 grant,
 and by the GWU REF grant.

\ \\ \ \\ \ \\
Department of Mathematics,\\
The George Washington University,\\
Washington, DC 20052\\
e-mail: {\tt przytyck@gwu.edu},\\
University of Maryland CP,\\
and University of Gda\'nsk\\
\ \\
Inst.. of Mathematics, University of Gda\'nsk,\\
e-mail: {\tt wrosicki@mat.ug.edu.pl}

\end{document}